\documentclass[11pt]{amsart}

\usepackage{amssymb,amsmath,enumerate,epsfig}
\usepackage{amsfonts,color}
\usepackage{a4,latexsym}
\usepackage{parskip}
\usepackage[all]{xy}
\usepackage{hyperref}
\hypersetup{pdftitle=dyn}

\newcommand{\C} {\mathbb{C}}
\newcommand{\Q} {\mathbb{Q}}

\newcommand{\Z}{\mathbb{Z}}

\newcommand{\PP}{\mathbb{P}}
\newcommand{\NS}{\mathop{\rm NS}}
\newcommand{\Num}{\mathop{\rm Num}}
\newcommand{\MW}{\mathop{\rm MW}}

\newcommand{\disc}{\mathop{\rm disc}}

\newcommand{\Pic}{\mathop{\rm Pic}}

\newcommand{\rk}{\mathop {\rm rk}}

\newcommand{\Jac}{\mathop{\rm Jac}}

\newtheorem{Theorem}{Theorem}[section]
\newtheorem{Proposition}[Theorem]{Proposition}
\newtheorem{Lemma}[Theorem]{Lemma}

\newtheorem{Corollary}[Theorem]{Corollary}

\theoremstyle{remark}

\newtheorem{Claim}[Theorem]{Claim}
\newtheorem{Remark}[Theorem]{Remark}
\newtheorem{Criterion}[Theorem]{Criterion}
\newtheorem{Observation}[Theorem]{Observation}
\newtheorem{Example}[Theorem]{Example}

\theoremstyle{definition}

\newtheorem{Definition}[Theorem]{Definition}

\begin{document}

\title{Moduli of Gorenstein $\Q$-homology  projective planes}


\author{Matthias Sch\"utt}
\address{Institut f\"ur Algebraische Geometrie, Leibniz Universit\"at
  Hannover, Welfengarten 1, 30167 Hannover, Germany}
\email{schuett@math.uni-hannover.de}

\thanks{Funding   by ERC StG~279723 (SURFARI) 
 is gratefully acknowledged.}

\date{\today}

\begin{abstract}
We give a complete classification of complex $\Q$-homology  projective planes
with numerically trivial canonical bundle.
There are 31 types, and each has one-dimensional moduli.
In fact, all moduli curves are rational and defined over $\Q$,
and we determine all families explicitly using 
extremal rational elliptic surfaces 
and Enriques involutions of base change type.
\end{abstract}
%
%
 \maketitle

 \section{Introduction}
 \label{s:intro}
 
 Fake projective planes continue to be among the
most fascinating objects of study in algebraic geometry,
starting from Mumford's construction in the 70's \cite{Mumford}
and culminating in the classification
by Prasad-Yeung \cite{PY} and Cartwright-Steger \cite{CS}.
Here one is solely concerned with smooth projective surfaces
whose (co)homology resembles $\PP^2$,
i.e. with Betti numbers $b_2=1, b_1=b_3=0$.

It is then natural to weaken the conditions
by allowing for normal projective surfaces $S$
whose $\Q$-homology equals that of $\PP^2$.
Typically, one considers Gorenstein $\Q$-homology projective planes with quotient singularities;
here we restrict to isolated rational double point singularities.
By definition, they are obtained from
smooth projective surfaces by contracting a collection of $(-2)$-curves.
For their understanding it is thus crucial to study the configurations of these curves.

One
 distinguishes whether $K_S$ is ample, numerically trivial or anti-ample.
 The latter case
 corresponds to log del Pezzo surfaces 
 and is well understood.
 Indeed, Ye gives a complete classification of  $\Q$-homology log del Pezzo surfaces in \cite{Ye}.
We point out that with one exception they are isolated in moduli.
  At the other end of the scale, the general type $\Q$-homology projective planes
  have resisted a thorough investigation so far.
 
This paper addresses the case where $K_S$ is numerically trivial.
Here the minimal resolution  is an 
 Enriques surface by the Enriques--Kodaira classification of complex algebraic surfaces.
 By \cite{HKO}, there are 31 configurations of $(-2)$-curves on Enriques surfaces
 whose contraction may give a $\Q$-homology projective surface,
see Theorem \ref{thm:31}.
  Here the moduli picture has not been determined so far:
   29 of these root types  are supported by examples in \cite{HKO}, 
(some isolated, some occurring in one-dimensional families),
one further type can be found in \cite{RS}.
This paper's main goal is the full clarification of the picture,
in particular with respect to moduli,
by means of a complete classification of Enriques surfaces
supporting the 31 maximal root types.

\begin{Theorem}
\label{thm}
For the 31 maximal root types realizable on Enriques surfaces,
the following hold
\begin{enumerate}[(i)]
\item
the root types are supported on 1-dimensional families of Enriques surfaces;
\item
the moduli spaces  can have up to 3 different components;
\item
each family has rational base and is defined over $\Q$;
\item
each family can be parametrised explicitly,
see Table \ref{T2}.
\end{enumerate}
\end{Theorem} 

%
%

The proof of Theorem \ref{thm} crucially relies on the observation
(in Proposition \ref{prop:base})
that each Enriques surface admits an elliptic fibration
which  arises from a base change type construction
following Kond\=o \cite{Kondo-Enriques} (see also \cite{HS}).
In fact, this is so constructive in nature
that it will allow us to carry out all parametrizations explicitly.

The paper is organised as follows.
Section \ref{s:latt} reviews  root types on Enriques surfaces and the underlying lattice theory
relevant for this paper.
In Section \ref{s:ell}, we relate this to elliptic fibrations
before introducing the main geometric technique
of Enriques involutions of base change type in Section \ref{s:tech}.
Section \ref{s:RES} reviews extremal rational elliptic surfaces
which will serve as the building blocks for the gluing method in Section \ref{s:glue}.
Sections \ref{s:fiber} and \ref{s:fam} then prepare for the computation
of the explicit families of Enriques surfaces which will be summarized in Section \ref{s:explicit},
proving most part of Theorem \ref{thm}.
The final section concerns the number of components of the moduli spaces,
thus completing the proof of Theorem \ref{thm}.
Throughout the paper, we will work out instructive examples in detail
to give a clear idea of the methods and techniques involved.

\subsubsection*{Conventions}
We work over $\C$ or any algebraically closed field $K$ of characteristic zero
(although most computations are carried out over $\Q$, in fact). 
All root lattices $A_n, D_k, E_l$ are taken to be negative-definite.
Orthogonal sums are indicated by a '$+$',
and likewise by $2R$ etc.
The notation $R(m)$ indicates the same abstract $\Z$-module as $R$,
but with intersection form multiplied by an integer $m$.

\section{Maximal root types on Enriques surfaces}
\label{s:latt}

By \cite{Ye} (see also the references therein)
there are 27 root lattices supported on
Gorenstein log del Pezzo surfaces of rank 1 
(and possibly on surfaces of general type).
This result was complemented in \cite{HKO}
where the case of trivial canonical class was studied.
Necessarily, in this case, the minimal desingularization 
is an Enriques surface $Y$,
and there are 31 possible root types:

\begin{Theorem}[Hwang-Keum-Ohashi]
\label{thm:31}
Let $S$ denote a Gorenstein $\Q$-homology projective plane
with trivial canonical class.
Then the singularity type of $S$ is one of the 31 following:

$A_9$, $A_8 + A_1$, $A_7 + A_2$, $A_7 + 2A_1$,  $A_6 + A_2 + A_1$,   $A_5 + A_4$, $A_5 + A_3 + A_1$, $A_5 + 2A_2$, $A_5 + A_2 + 2A_1$, $2A_4 + A_1$, $A_4 + A_3 + 2A_1$, $3A_3$,  $2A_3 + A_2+ A_1$, 
$2A_3 + 3A_1$, $A_3 + 3A_2$,
$D_9$, $D_8 + A_1$, $D_7 + 2A_1$, $D_6 + A_3$, $D_6 + A_2 + A_1$, $D_6 + 3A_1$, $D_5 + A_4$, $D_5 + A_3 + A_1$, $D_5 + D_4$,  $D_4 + A_3 + 2A_1$, $2D_4 + A_1$,
$E_8 + A_1$, $E_7 + A_2$, $E_7 + 2A_1$, $E_6 + A_3$, $E_6 + A_2 + A_1.$
\end{Theorem}

We point out that except for the root types $2A_3 + A_2+ A_1$ and $A_3 + 3A_2$,
each type was supported by an example in \cite{HKO} 
-- most of them isolated, but some also occurring in 1-dimensional families.
In parallel, the type $A_3 + 3A_2$ is supported on an isolated example in \cite[Ex. 3.8]{RS}.
Theorem \ref{thm} establishes the full classification of Gorenstein $\Q$-homology projective plane
with trivial canonical class.
Not only does it provide examples supporting the remaining type $2A_3+A_2+A_1$,
but it also confirms the natural conjecture (conf.~Lemma \ref{lem:1-dim}) 
that each type is supported on one-dimensional families
(thus correcting an erroneous claim for the type $2A_3+3A_1$ from \cite{Keum-Viet}).
We emphasize the contrast with K3 surfaces:
over $\C$, K3 surfaces can support root types of maximal rank 19,
but those surfaces only appear in one-dimensional analytic families
whose algebraic members are isolated, yet countably infinite in number (compare \cite{xiao}).

For later use, we continue to elaborate the lattice theoretic background of Theorem \ref{thm:31}.
Recall that for an Enriques surface $Y$,
the canonical divisor encodes the torsion in $\Pic(Y)$.
Quotienting out by numerical equivalence,
we obtain as the free part of $\Pic(Y)$ the following unimodular even lattice:
\begin{eqnarray}
\label{eq:Num}
\Num(Y) \cong U + E_8
\end{eqnarray}
where $U$ denotes the hyperbolic plane.
Assume that $Y$ contains a configuration of nine smooth rational curves
which corresponds to an orthogonal sum $R$ of root lattices of total rank $9$.
Then this gives an embedding
\[
R\hookrightarrow \Num(Y)
\]
with orthogonal complement generated by a single positive vector $h$.
More precisely, we find that
$h^2=d>0$ where $-d$ is the discriminant of the primitive closure $R'$ of $R$ in $\Num(Y)$:
\[
R' = (R\otimes\Q)\cap\Num(Y).
\]
Following Nikulin \cite{Nikulin}, 
$R'$ and $\Z h$ glue together to the unimodular even lattice $\Num(Y)$ from \eqref{eq:Num}
by means of an isomorphism of discriminant groups
\[
A_{R'} = (R')^\vee/R' \cong A_{\Z h} \cong \Z/d\Z,
\]
which extends to an isomorphism of the discriminant forms taking values in $\Q/2\Z$
\begin{eqnarray}
\label{eq:q}
q_{R'}\cong -q_{\Z h}.
\end{eqnarray}
In the present situation, this details as follows:

\begin{Lemma}
\label{lem:R}
The rank $9$ root lattice $R$ may arise from a configuration of smooth rational curves on an Enriques surface
only if it admits an even finite index overlattice $R'$ such that
\[
q_{R'} \cong \Z/d\Z(-1/d).
\]
\end{Lemma}

\begin{Remark}
Theorem \ref{thm} implies that the condition from Lemma \ref{lem:R} is not only necessary,
but also sufficient as soon as $R$ has at most 4 orthogonal summands or $R=2A_3+3A_1$.
\end{Remark}

Note that  the even finite index overlattices of $R$ are encoded in the discriminant form $q_R$ as well.
Hence we can retrieve all information necessary to check the condition of Lemma \ref{lem:R}
from the discriminant form $q_R$.
In \cite{HKO}, the authors decided to rather argue with local epsilon-invariants,
but for our purposes it will be more beneficial to work with discriminant forms.
Before illustrating this with two instructive examples
(which will keep on occurring throughout the paper),
we fix notation for discriminant forms of root lattices.
All of this is well-known;
for instance, it is encoded in the correction terms of the height pairing
on Mordell-Weil lattices \cite{ShMW},
but it can also be computed directly without difficulty.

\subsection{Root lattice $A_n \; (n>0)$}

We number the vertices of the corresponding Dynkin diagram $a_1,\hdots,a_n$.

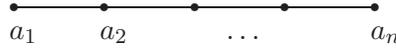
\begin{figure}[ht!]
\setlength{\unitlength}{.6mm}
\begin{picture}(80,10)(0,25)
\multiput(3,32)(20,0){5}{\circle*{1.5}}
\put(3,32){\line(1,0){80}}
\put(2,25){$a_1$}
\put(22,25){$a_2$}
\put(50,25){$\hdots$}
\put(82,25){$a_n$}
%
%

\end{picture}
\end{figure}

Then each vertex $a_i$ defines a unique dual vector
\[
a_i^\vee\in A_n^\vee \;\;\; \text{ such that } \;\;\; a_i^\vee.a_j = \delta_{ij}.
\]
Here
\[
(a_i^\vee)^2 = -\frac{i(n+1-i)}{n+1} \;\;\; \text{ and } \;\;\; A_n^\vee/A_n = \langle a_1^\vee\rangle \cong \Z/(n+1)\Z
\]
with the property that $a_m^\vee \equiv m\cdot a_1^\vee \mod A_n$.

\subsection{Root lattice $D_k \; (k\geq 4)$}

We number the vertices of the corresponding Dynkin diagram $d_1,\hdots,d_k$.

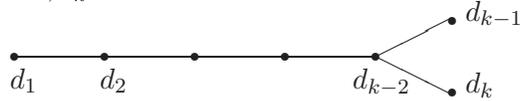
\begin{figure}[ht!]
\setlength{\unitlength}{.6mm}
\begin{picture}(100,20)(5,-5)
%
\multiput(3,8)(20,0){5}{\circle*{1.5}}
\put(3,8){\line(1,0){80}}
\put(83,8){\line(2,1){17}}
\put(83,8){\line(2,-1){17}}
\put(100,16){\circle*{1.5}}
\put(100, 0){\circle*{1.5}}
\put(2,1){$d_1$}
\put(22,1){$d_2$}
\put(78,1){$d_{k-2}$}
\put(103,16){$d_{k-1}$}
\put(103, 0){$d_{k}$}


\end{picture}
\end{figure}

Then each vertex $d_i$ defines a unique dual vector
\[
d_i^\vee\in D_k^\vee \;\;\; \text{ such that } \;\;\; d_i^\vee.d_j = \delta_{ij}.
\]
Here
\[
(d_1^\vee)^2 = -1   \;\;\; \text{ and } \;\;\; (d_{k-1}^\vee)^2 = (d_k^\vee)^2 = -\frac k4
\]
with $d_1^\vee.d_{k-1}^\vee = d_1^\vee.d_{k}^\vee = -1/2$ 
and 
$d_{k-1}^\vee.d_{k}^\vee = -(k-2)/4$.
One has
\[
D_k^\vee / D_k =
\begin{cases}
\langle d_k^\vee\rangle \cong \Z/4\Z & \text{ if $k$ is odd,}\\
\langle d_1^\vee, d_k^\vee\rangle \cong (\Z/2\Z)^2 & \text{ if $k$ is even.}
\end{cases}
\]

\subsection{Root lattice $E_l \; (l=6,7,8)$}
\label{ss:E_l}

We number the vertices of the corresponding Dynkin diagram $e_1,\hdots,e_l$.

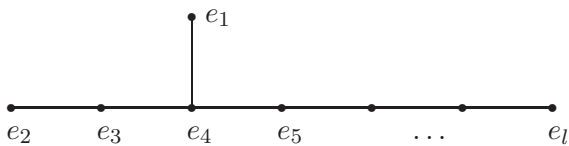
\begin{figure}[ht!]
\setlength{\unitlength}{.6mm}
\begin{picture}(120,30)(3,-40)
%
\multiput(3,-32)(20,0){7}{\circle*{1.5}}
\put(3,-32){\line(1,0){120}}
\put(2,-39){$e_2$}
\put(22,-39){$e_3$}
\put(42,-39){$e_4$}
\put(62,-39){$e_5$}
\put(43,-32){\line(0,1){20}}
\put(43,-12){\circle*{1.5}}
\put(46,-13){$e_1$}
\put(92,-39){$\hdots$}
\put(122,-39){$e_l$}

\end{picture}
\end{figure}

There is a unique dual vector
\[
e_l^\vee\in E_l^\vee \;\;\; \text{ such that } \;\;\; e_l^\vee.e_j = \delta_{lj}.
\]
Here
\[
(e_l^\vee)^2 = -\frac{10-l}{9-l} \;\;\; \text{ and } \;\;\; E_l^\vee/E_l = \langle e_l^\vee\rangle\cong \Z/(9-l)\Z.
\]

\begin{Example}[$R=A_9$]
\label{ex:1}
Take the root lattice $R=A_9$ (the first from the list in Theorem \ref{thm:31}).
Since $R$ has square-free discriminant $-10$,
there are no integral overlattices.
Hence we have to check whether 
\[
A_9 + \langle 10\rangle
\]
glue together to the even unimodular lattice $U+E_8$.
The generators of the discriminant groups have square
$(a_1^\vee)^2=-9/10$ and $1/10$.
Since these agree up to sign and a square factor (relatively prime to the order $10$),
the gluing is easily achieved.
\end{Example}

\begin{Example}[$R=A_8+A_1$]
\label{ex:2}
Take the root lattice $R=A_8+A_1$ (the second from the list in Theorem \ref{thm:31}).
The only integral overlattice is obtained by adjoining the $(-2)$-vector $a_3^\vee\in A_8^\vee$ to $A_8$.
Necessarily this leads to the unimodular lattice $E_8$.
We discuss two cases:

If $R'=E_8+A_1$, then $h^2=2$ and the gluing of $A_1$ and $\Z h$ to $U$ is immediate.

If $R=R'=A_8+A_1$, then we require $h^2=18$.
A generator of $A_R$ is given by the sum of both $a_1^\vee$'s for $A_8$ and $A_1$.
It has square $-8/9-1/2=-25/18$ which agrees up to sign and a square (invertible modulo $18$) with $(h/18)^2=1/18$.

In conclusion, according to Lemma \ref{lem:R}, both cases may a priori occur on an Enriques surfaces.
We will see in \ref{ss2:A_8+A_1} that this indeed gives rise to two distinct cases in practice.
\end{Example}

%
%
%
%

\section{Elliptic fibrations on Enriques surfaces}
\label{s:ell}

It is one of the key features of Enriques surfaces $Y$
that they always admit elliptic fibrations
\begin{eqnarray}
\label{eq:ell}
f:\; Y \to \PP^1,
\end{eqnarray}
and in fact often many.
Necessarily they come with two fibers of multiplicity two
(whose supports are usually called half-pencils)
and with some bisection
(so strictly speaking, the generic fiber is a genus one curve,
but not endowed with a rational point over $K(\PP^1)$).
There are (at least) two ways to see this.
On the one hand,
consider the universal covering $X$ which is a K3 surface:
\[
\pi: X \to Y
\]
Naturally the elliptic fibration carries over to $X$,
and it is compatible with the fixed-point free involution $\imath$
whose quotient recovers $Y$.
Here $\imath$ also acts on the base $\PP^1$ of the fibration,
and the multiple fibers of $Y$ sit over the fixed points.

Alternatively, 
consider the jacobian fibration
\[
\Jac(Y) \to \PP^1.
\]
This is a rational elliptic surface
with the same singular fibers as $Y$,
but of course with a section.
Reversely, starting from $\Jac(Y)$, one can head in two different directions.
First, the original surface $Y$
can be recovered analytically by a logarithmic transformation.
Secondly, one can apply a quadratic base change ramified in the two fixed point of $\imath$ on $\PP^1$.
This gives rise to an elliptic K3 surface with section --  the jacobian of $X$.

What makes elliptic fibrations on K3 and Enriques surfaces so useful
is their compatibility with lattice theory.
Most prominently, any non-zero divisor $D$ 
on a K3 surface $X$ of square $D^2=0$ gives rise to an elliptic fibration 
after subtracting the base locus from $D$ or $-D$, whichever is effective
(cf.~\cite{PSS}),
and any root lattice $R\subset\NS(X)$ perpendicular to $D$ is supported on fiber components 
(compare \cite{Kondo-Aut} and Prop.~\ref{prop:E}, Lem.~\ref{prop:fibers}).

On Enriques surfaces,
a similar relation persists,
but with a striking difference:
on a K3 surface $X$, for any divisor $D$ of square $D^2=-2$, either $D$ or $-D$ is effective by Riemann-Roch.
Consequently it supports a connected configuration of smooth rational curves,
and root lattices inside $\NS(X)$  always come from smooth rational curves.
On an Enriques surface $Y$, however, this is far from being true.
This can already be seen in the decomposition $\Num(Y)=U+E_8$:
clearly $E_8$ is supported on roots, i.e.~classes of square $-2$,
but a general Enriques surface contains no smooth rational curves at all \cite{BP}.
We record the following important lemma concerning divisibility among root lattices:

\begin{Lemma}
\label{lem:-2}
Let $R$ be a root lattice supported on an Enriques surfaces $Y$
with generators represented by $(-2)$-curves.
If $R\neq R'\subset\Num(Y)$,
then the additional $(-2)$-vectors are neither effective nor anti-effective.
\end{Lemma}

\begin{proof}
This is just like on a K3 surface
where divisibilities among $(-2)$-curves start at configurations such as $8A_1$ and $6A_2$,
compare \cite{Barth}, for instance.
(The only difference is that here, due to failure of Riemann-Roch, $R$ may still be divisible.)
\end{proof}

This lemma will play a central role in our classification of Enriques surfaces supporting root types of rank $9$,
and in particular for the computation of the different moduli components.
Before getting there, we emphasize an instructive example.

\begin{Example}[$R=A_8+A_1$ cont'd]
\label{ex:2-2}
For the root type $R=A_8+A_1$, 
we distinguished two cases in Example \ref{ex:2}.
If $R$ is primitive in $\Num(Y)$, i.e. $R=R'$,
then it is supported on smooth rational curves by assumption.
If $R$ is not primitive, i.e. $R'=E_8+A_1$,
then it is obtained from $R$ by adding the $(-2)$-vector $a_3^\vee$
which by Lemma \ref{lem:-2} is neither effective nor anti-effective.
%
\end{Example}


In the opposite direction, consider an Enriques surface containing a smooth rational curve $C$.
It is immediate that $\pi^*C$ splits into two disjoint smooth rational curves on $X$.
That is, by the above discussion, smooth rational curves on Enriques surfaces are duplicated on the K3 cover.
We directly obtain the following corollary
which shows that Theorem \ref{thm} is the best we can hope for
if it comes to moduli dimensions:

\begin{Lemma}
\label{lem:1-dim}
Enriques surfaces supporting a  root type $R$ of rank $9$ come in 1-dimensional families at most.
\end{Lemma}

\begin{proof}
Moduli of Enriques surfaces are governed by the covering K3 surfaces.
Denote the Enriques surface with root type $R$ of rank $9$ by $Y$
and the K3 cover by $X$.
By what we have said before,
we have an embedding
\[
2R \hookrightarrow \NS(X).
\]
On the other hand, $X$ is automatically algebraic as there is another embedding
\[
U(2)+E_8(2) \cong \pi^*\Num(Y) \hookrightarrow \NS(X).
\]
Since $R^2$ is negative-definite of rank $18$ while $\Num(Y)$ is hyperbolic,
we find that $\rho(X)\geq 19$. 
It is known from the theory of lattice polarized K3 surfaces,
that such K3 surfaces come in 1-dimensional families.
\end{proof}

We are now in the position to record the properties of elliptic fibrations
on Enriques surfaces relevant for our purposes.
Beforehand, we point out another important difference to K3 surfaces
next to the subtleties involving $(-2)$-classes:
for a half-pencil $E$, only $|2E|$ will induce an elliptic fibration, but not $|E|$.

\begin{Proposition}
\label{prop:E}
Let $Y$ be an Enriques surface
and $E$ a divisor with non-zero class in $\Num(Y)$ and square $E^2=0$.
Then
\begin{enumerate}[(i)]
\item
$E$ gives rise to an elliptic fibration on $Y$;
\item
any smooth rational curve $C$ perpendicular to $E$
gives an effective or anti-effective $(-2)$-divisor supported on fiber component of the fibration;
\item
in particular, any root lattice $R\subset\Num(Y)$
which is perpendicular to $E$ and 
supported on $(-2)$-curves on $Y$,
embeds into the singular fibers of the fibration (as a lattice);
\item
if $Y$ contains a smooth rational curve $B$ with $B.E=1$,
then $E$ gives a half-pencil, and $B$ gives a bisection.
\end{enumerate}
\end{Proposition}

\begin{proof}
\textit{(i)}
By Riemann-Roch, either $E$ or $-E$ is effective,
say w.l.o.g.~$E$.
If its class in $\Num(Y)$ is not primitive,
then divide by the degree of primitivity
and continue with the primitive class thus obtained.
By \cite[(17.4)]{BHPV},
it remains to subtract the base locus of $|2E|$;
this is given by a configuration of  $(-2)$-curves.
Hence it suffices to apply successive reflections in these curves
to arrive at a half-pencil.

\textit{(ii), (iii)}
We argue with the procedure sketched above to derive a half-pencil from $E$.
The only critical step consists in the reflections
since these generally do not act effectively on $\Num(Y)$.
On the classes of $(-2)$-curves, however, they do act effectively up to sign.
This follows from the fact that the $(-2)$-curves split on the K3 cover $X$,
and likewise that a reflection on $Y$ corresponds to a composition of reflections on $X$
(compare the discussion in \cite[\S 2.3, 2.4]{RS}).
Thus we can assume that the smooth rational curve $C$, or the root lattice $R$ supported on $(-2)$-curves on $Y$,
are mapped to (anti-)effective $(-2)$-divisors 
perpendicular to the half-pencil $E$.
That is, their classes embed into $E^\perp$ 
(and into $E^\perp/\Z E$),
and we conclude as in \cite[Lem. 2.2]{Kondo-Aut}
that the $(-2)$-divisors are supported on fiber components
(but unlike loc. cit. we do not claim that they determine the singular fibers,
since we have no assumption such as $E^\perp/\Z E=R$, see Lemma \ref{prop:fibers}).

\textit{(iv)}
By inspection of the intersection number, $E$ has primitive class in $\Num(Y)$.
We continue as in \textit{(i)}, using the special action of the reflections on $(-2)$-curves
as laid out in \textit{(ii), (iii)}.
If $E$ is effective, then  the reflections map $B$ to an effective $(-2)$-divisor $D$
such that $D.E'=1$.
Since $E'$ is nef
(as is any fiber in a fibration by the moving lemma),
and $D$ is supported on $(-2)$-curves,
we infer that the support of $D$ contains a smooth rational bisection $B'$
with $B'.E'=1$
plus possibly some configuration of $(-2)$-curves.
Since the latter are orthogonal to the half-pencil $E'$ by construction, they have to be fiber components,
and successive reflections map $B$ to $B'$ while not affecting $E'$.

If $-E$ is effective, then $B$ is contained in the support of $-E$.
The reflection in $B$ gives an effective isotropic class $E'=-E-B$
such that $B.E'=1$. Now continue as before to complete the proof of Proposition \ref{prop:E}.
\end{proof}

\begin{Remark}
\label{rem:curves?}
In the generality of Proposition \ref{prop:E},
we cannot claim that the $(-2)$-curves in \textit{(ii)} and \textit{(iii)} correspond to proper fiber components
(a subtlety pointed out by a referee of \cite{RS}).
However, in all the situations relevant to our concerns,
this will in fact crucially hold true as we shall see along the way of proving Theorem \ref{thm} (see Proposition \ref{prop:curves}).
\end{Remark}

We emphasize the most instructive case of Proposition \ref{prop:E}:
when $E$ itself happens to be a divisor of Kodaira type,
then it naturally appears as a singular fiber of the induced fibration (potentially multiple, for instance in case \textit{(iv)}).
For later use, we continue to elaborate on the relation between the root lattice $R$ in Proposition \ref{prop:E}  \textit{(iii)}
and the underlying singular fibers.
To this end, we recall that the fibers correspond (non-uniquely) to the extended Dynkin diagrams $\tilde A_n, \tilde D_k, \tilde E_l$.

\begin{Lemma}
\label{prop:fibers}
Let $E$ be a non-trivial isotropic vector in $\Num(Y)$
and $R$ a root lattice supported on $(-2)$-curves on $Y$.
Let $R_0= E^\perp \cap R$ decompose into irreducible root lattices
\[
R_0 = \sum_v R_v.
\]
Then the following are equivalent:
\begin{enumerate}
\item
 the singular fibers of the elliptic fibration induced by $E$ 
support a root lattice strictly greater than $R_0$;
\item
the reducible fibers are not in bijective correspondence to the extended Dynkin diagrams $\tilde R_v$;
\item
$R_0$ is a proper sublattice of the span of the $(-2)${-curves} inside $E^\perp/\Z E$. 
\end{enumerate}
\end{Lemma}

\begin{proof}
The equivalence of \textit{(1)} and \textit{(2)} follows from Kodaira's classification:
$\tilde R_v$ gives the singular fiber $F_v$ if and only if in $R_v$ a simple fiber component is omitted (cf.~Rem.~\ref{rem:unique}).
Otherwise $R_v$ will have finite index in the root lattice corresponding to the fiber $F_v$.
The same applies to a sum of irreducible root lattices embedding into a single fiber.

\textit{(3)} is exactly how Kond\=o's argument in \cite[Lem. 2.2]{Kondo-Aut} has to be rephrased
if  $E^\perp/\Z E$ is not a priori a root lattice (both on Enriques and K3 surfaces):
the singular fibers are encoded in the root lattice inside $E^\perp/\Z E$
generated by  $(-2)$-curves on $Y$ perpendicular to $E$.
Hence the equivalence with \textit{(1), (2)} follows.
\end{proof} 

%
%
%
%
%
%

In the next section,
we start translating the lattice-theoretic information from this section
into explicit equations of the elliptic fibrations.
For this purpose, it will be instrumental to 
determine how the bisection intersects the singular fibers.
The subtleties in this respect have two facets:
$R_0$ may not determine the singular fibers (as in Lemma \ref{prop:fibers}),
and there may be multiple fibers.
On the other hand, Lemma \ref{lem:-2} provides a severe restriction on the possible configurations.
We illustrate this with an easy example continuing a previous thread:

\begin{Example}
\label{ex:ae8}
We have already seen in Example \ref{ex:2} that the root lattice $A_8$ embeds into $E_8$.
However, the divisible class necessarily is effective or anti-effective,
so in the set-up of Proposition \ref{prop:E} \textit{(iii)} this would not be compatible with Lemma \ref{lem:-2}.

In contrast, the corresponding fiber $\tilde E_8$ admits an embedding of $A_8$, even as $(-2)$-curves labelled 
$e_2,\hdots,e_9$, which does not cause any divisibility (the $3$-divisibility arises only mod $E$!).
Note that these considerations do not involve a smooth bisection yet
which, in our cases, we will play off soon against the restrictions posed by Lemma \ref{lem:-2}.
\end{Example}


For later reference, we sum up how we will piece the given information together,
using all the restrictions and findings above:

\begin{Observation}
\label{lem:bi}
Let $Y$ be an Enriques surface with root type $R$ of rank $9$ 
supported on $(-2)$-curves.
Assume that there is an isotropic class $E\in\Num(Y)$
such that $R_0=E^\perp\cap R$ is a root lattice of rank $8$.
In addition, assume that the support of the root lattice $R$ 
also contains the class of an smooth rational curve $B$ 
such that $B.E=1$.
Then the  bisection $B'$ of the elliptic fibration arising from $B$ by Proposition \ref{prop:E}
meets the $(-2)$-divisors resulting from $R_0$ in a way exactly predicted by $R$,
and usually this can be translated into the intersection pattern with the fiber components.
\end{Observation}

The reader may wonder about the phrase 'usually' above,
and indeed this will hold true in all cases under consideration here (see Proposition \ref{prop:curves}),
but proving the statement in full generality would go beyond the scope of this paper
(compare Remark \ref{rem:curves?}).
Here we content ourselves by  illustrating this by an instructive example.


\begin{Example}[$R=A_8+A_1$ again]
\label{ex:2-3}
Take $R=A_8+A_1$.
Following up on Example \ref{ex:2-2},
consider the imprimitive case where $R'=E_8+A_1$.
The gluing from Example \ref{ex:2}
gives rise to the isotropic vector $E=a_1^\vee+h/2$ (for $a_1^\vee\in A_1^\vee$).
By Proposition \ref{prop:E} the $(-2)$-curve $B$ generating the $A_1$-summand induces a bisection $B'=\sigma (B)$
of the  elliptic fibration $|2 \sigma (E)|$ for a suitable composition of reflections $\sigma$.
We have $R_0=A_8$,
admitting embeddings into singular fibers of types $\tilde A_8$ and $\tilde E_8$.
In either case, it is clear that the support of $\sigma (R_0)$ does not contain the fiber components met by $B'$.
For rank reasons, $B'$ thus meets a single component,
and $\sigma (R_0)$ embeds into its complement, i.e.~into $A_8$ resp.~$E_8$.
In the former case, the general theory of root lattices guarantees
that there are further reflections taking $\sigma (R_0)$ to the canonical generators of $A_8$;
i.e.~all $(-2)$-curves forming $R$ are mapped to $(-2)$-curves again (as announced in Remark \ref{rem:curves?},
see Proposition \ref{prop:curves}).
In the latter case, with 
\[
A_8\cong \sigma (R_0)\hookrightarrow E_8,
\]
the $(-2)$-vector $\sigma (a_3^\vee)$ necessarily becomes effective or anti-effective.
Thus the same holds for $a_3^\vee$, contradicting Lemma \ref{lem:-2}.
Hence this case cannot persist,
and the reducible fiber can only have type $\tilde A_8$ (under the assumption that $R$ is not primitive).
%
\end{Example}

We shall now turn to the task of extracting explicit parametrizations from the  data provided by the root types.
Before doing so, let us highlight that these subtle relations between $(-2)$-classes and curves,
and between singular fibers and root lattices
are the main reason why it is  non-trivial to decide
whether a root lattice $R$ from Theorem \ref{thm:31} is supported on an Enriques surface
-- or phrased positively, why $R$ may even admit several moduli components.


%

\section{Enriques involutions of base change type}
\label{s:tech}

This section introduces the main theoretical ingredient
for the proof of Theorem \ref{thm}.
We review a construction of Enriques involutions from \cite{HS}
based on quadratic twists
which can be controlled very well and is nicely compatible with 
the theories of lattice polarized K3 surfaces
and Mordell-Weil lattices.

The overall set-up is as follows:
assume that an elliptic fibration \eqref{eq:ell} on an Enriques surface $Y$
admits a bisection $C$
which splits on the K3 cover $X$ into two smooth rational curves
(not necessarily disjoint).
The induced elliptic fibration on $X$ is thus equipped with two sections
one of which, say $O$, we choose as zero for the group law on the fibers.
In particular,
$X=\Jac(X)$ and we have a commutative diagram consisting of elliptic fibrations
and $2:1$ maps:
$$
\begin{array}{ccccl}
&& X &&\\
& \swarrow &  \downarrow & \searrow &\\
Y && \PP^1 && S=\Jac(Y)\\
\downarrow & \swarrow && \searrow & \downarrow\\
\PP^1 & \multicolumn{3}{c} = & \PP^1
\end{array}
$$
Let $\jmath$ denote the deck transformation corresponding to the double covering $X\to S$.
Then one can show that the other section mapping to $C$,
say $P$, is anti-invariant in $\MW(X)$ for $\jmath^*$,
and that the Enriques involution $\imath$ equals $\jmath$ composed with translation by $P$
\cite[Prop.~3.1]{HS}.

The crucial point of this construction is that 
the shape of $P$ can be predicted precisely,
so that often $P$ (and everything else) can be computed explicitly.
Notably, $P$ is invariant for the composition of $\jmath$ and $-1$ (acting fiberwise),
hence it descends to the quotient of $X$ by $\jmath\circ(-1)$.
The latter automorphism of $X$ is a Nikulin involution with four fixed points  in each fiber fixed by $\jmath$,
so the minimal desingularization 
\[
S'\to X/\langle\jmath\circ(-1)\rangle
\]
 is again K3.
In fact, this is exactly the quadratic twist of $S$ at the two fixed point of $\jmath$ (or $\imath$)
acting on $\PP^1$.
Explicitly,
if $S$ is given in (extended) Weierstrass form
\begin{eqnarray}
\label{eq:S}
S:\;\;\; y^2 = f(x),\;\;\; f\in K(t)[x],\; \;\; \deg f=3,
\end{eqnarray}
then the quadratic twist at the zeroes of a quadratic polynomial $q\in K[t]$ (possibly linear, i.e.~including a zero at $\infty$)
is given by
\begin{eqnarray}
\label{eq:S'}
S': \;\;\; q(t)y^2 = f(x).
\end{eqnarray}
Here both $S$ and $S'$ pull-back to $X$ under the degree two morphism of $\PP^1$ ramified at the roots of $q(t)$,
and the section $P$ is induced by a section $P'\in\MW(S')$.
Basically, the only condition for the composition of $\jmath$ and translation by $P$
to be fixed point-free is that $P'$ meets both ramified fibers on $S'$
in a component far away from the identity component (the one met by the zero section
which is chosen compatibly with $O$;
if the fiber has type $I_n^*$ with $n>0$,
i.e.~when the ramified fiber on $Y$ is not smooth,
then this has to be one of the two far simple components).

While the above technique is extremely useful
for explicit computations and constructions,
it is in fact not so easy to detect on the level of Enriques surfaces
when it applies.
That is, of course, unless the bisection $C$ itself is smooth rational
because then it splits automatically on $X$ as discussed in Section \ref{s:ell}.

\begin{Lemma}
\label{lem:bct}
Let $Y$ be an Enriques surface with given elliptic fibration admitting a smooth rational bisection.
Then $Y$ arises from an Enriques involution of base change type
originating from $\Jac(Y)$.
\end{Lemma}

This basic case
which originally goes back to Kond\=o \cite{Kondo-Enriques}
before being generalized in \cite{HS}
will be instrumental for our paper,
so we continue to elaborate on the details.

Assuming that the bisection $C$ is smooth rational as in Lemma \ref{lem:bct} (i.e. $C^2=-2$),
it splits on $X$ into orthogonal sections $O, P$.
In other terminology, $P$ is an \emph{integral} section (relative to the neutral element $O$ in $\MW(X)$).
One can easily translate this into the shape of $P'$.
To this end, we make sure that the extended Weierstrass forms involved are minimal.
In terms of \eqref{eq:S}, \eqref{eq:S'},
one can assume after some variable transformations that
\[
f = x^3 + a_2 x^2 + a_4 x + a_6, \;\;\; a_i\in K[t].
\]
Then the minimality of \eqref{eq:S} is equivalent to the 
condition that 
\[
\deg a_i \leq i \;\;\; (i=2,4,6), 
\]
but at the same time
the discriminant of $f$ should not be a perfect twelfth power in $K[t]$.
In consequence, the integral sections of $S$ 
(which always generate the Mordell-Weil group by \cite{OS})
are given by pairs of polynomials
$(U,V)$ in $t$ such that $\deg U\leq 2, \deg V\leq 3$.
In what follows,
we will abbreviate these polynomials (or in general rational functions)
by $x(P), y(P)$.

Accordingly, on $S'$ we shall concentrate on integral sections $P'$
meeting both ramified fibers in non-identity components.
In terms of \eqref{eq:S'},
these conditions translate as
\begin{eqnarray}
\label{eq:x(P')}
x(P'), y(P') \in K[t] \;\;\; \text{ with } \;\;\; 
\deg x(P'), \deg y(P') \leq 2.
\end{eqnarray}
In practice,
given $S$ and $q$,
this approach  leads to a system
of polynomial equations
which often can be solved directly or with a computer algebra system.
It turns out to be even more beneficial 
for computing families of Enriques surfaces and their covering K3 surfaces
as then we allow $q$ to vary as well.
An example in this spirit was given for the generic nodal Enriques surface in \cite[3.8]{HS}.
Here we illustrate this with a 1-dimensional family geared towards our aims.

\begin{Example}[$R=A_9$ cont'd]
\label{ex:1-2}
Consider one of the (so-called extremal) rational elliptic surfaces with finite Mordell-Weil group $\MW\cong\Z/4\Z$:
\[
X_{8211}: \;\;\; y^2 = x(x^2+(t^2+2)x+1).
\]
This has singular fibers of Kodaira types $I_8$ at $\infty, I_2$ at $0$ and $I_1$ at the roots of $t^2+4$.
There are 4-torsion sections $(-1,\pm t)$ and a 2-torsion section $(0,0)$.
For any quadratic twist, the 2-torsion section is simultaneously invariant and anti-invariant  for the deck transformation $\jmath^*$,
so it leads to an Enriques involution if and only if there is no ramification at the $I_2$ fiber.
This leads to a 2-dimensional family of Enriques surfaces with root types, e.g., $E_7+A_1$ and $A_7+A_1$,
but by Lemma \ref{lem:1-dim}, it cannot  support as a whole any of the types from Theorem \ref{thm:31}.

In order to support Enriques surfaces of root type $A_9$,
we arrange for the quadratic twist at $q$
to admit a section $P'$ meeting both $I_8$ and $I_2$ fibers at non-identity components
adjacent to the identity component.
This translates as 
\[
x(P')|_{t=0} = -1,\;\;\; x(P')|_{t=\infty} = 0
\]
(where the last equality should be understood as a simple vanishing).
By \eqref{eq:x(P')}, this directly leads to
\[
x(P') = \mu t -1
\]
which for $\mu\neq 0$ lives exactly on the quadratic twist $S'$ at 
\[
q= (\mu t-1) (\mu t +\mu^2-1).
\]
The base change type technique thus gives a one-dimensional family of Enriques surfaces
(properly parametrized by $\sqrt \mu$ after eliminating symmetries)
with reducible singular fibers of Kodaira types $I_8, I_2$ and a smooth rational bisection $C$
(the common image of $O$ and $P$).
The Enriques surfaces (for $\mu\neq 0$) feature the following diagram of $(-2)$-curves
which obviously support the root type $A_9$:

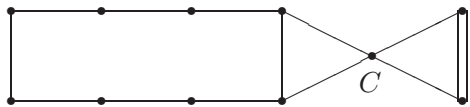
\begin{figure}[ht!]
\setlength{\unitlength}{.6mm}
\begin{picture}(100,20)(0,0)
%
\multiput(0,0)(20,0){4}{\circle*{1.5}}
\put(0,0){\line(1,0){60}}

\multiput(0,20)(20,0){4}{\circle*{1.5}}
\put(0,20){\line(1,0){60}}

\put(0,20){\line(0,-1){20}}
\put(60,20){\line(0,-1){20}}

\put(60,20){\line(2,-1){40}}
\put(60,0){\line(2,1){40}}

\put(99,20){\line(0,-1){20}}
\put(101,20){\line(0,-1){20}}

\put(80,10){\circle*{1.5}}
\put(77,2){$C$}

\put(100,20){\circle*{1.5}}
\put(100,0){\circle*{1.5}}


\end{picture}
\caption{$(-2)$-curves supporting the root type $A_9$}
\label{fig:A_9}
\end{figure}
\end{Example}

In the next two sections we shall see how all Enriques surfaces
supporting root types of rank $9$ (thus from Theorem \ref{thm:31})
arise naturally through the base change construction,
and in fact from very special rational elliptic surfaces (see Proposition \ref{prop:base}).

\section{Extremal rational elliptic surfaces}
\label{s:RES}

The aim of the next two sections is to prove a  classification result which
will be crucial for the proof of Theorem \ref{thm}.
In essence, we will narrow down the rational elliptic surfaces
required to study
Enriques surfaces
supporting root types of rank $9$
to the following very special class:

\begin{Definition}
\label{def}
A rational elliptic surface $S\to \PP^1$ is called \emph{extremal}
if its Mordell-Weil group $\MW(S)$ is finite.
\end{Definition}

Note that we have already seen an example of an extremal rational elliptic surface
in Example \ref{ex:1-2}.
Extremality translates into singular fibers and lattices as follows:

\begin{Criterion}
\label{crit}
A rational elliptic surface is extremal if and only if 
the singular fibers support  root lattices of total rank $8$.
\end{Criterion}

\begin{proof}
Following Kodaira's classification \cite{K},
the reducible fibers are associated with root lattices $R_v$
of rank one less than the number of fiber components
(compare Lemma \ref{prop:fibers}).
Writing $R=\oplus_v R_v$,
the Shioda-Tate formula reads
\begin{eqnarray}
\label{eq:ST}
\rho(S) = 2 + \rk R + \rk \MW(S).
\end{eqnarray}
Recall that any rational elliptic surface $S$ is isomorphic to $\PP^2(9)$,
the projective plane blown up in the base points of a cubic pencil.
In particular, $\rho(S)=10$,
and finite $\MW$ is equivalent to $\rk R = 8$ by \eqref{eq:ST} as claimed.
\end{proof}

\begin{Remark}
\label{rem:unique}
In associating a root lattice $R_v$ to the singular fiber at $v$,
one usually omits the fiber component met by the zero section.
This choice is particularly compatible with the theory of Mordell-Weil lattices \cite{ShMW},
but for additive fibers, it is by no means unique
as we have seen in Section \ref{s:ell}, especially in Lemma \ref{prop:fibers} and Example \ref{ex:ae8}.
For instance, omitting a component of an $I_0^*$ fiber, we either obtain $D_4$ or $4A_1$.
\end{Remark}


 We point out that there are only 16 extremal rational elliptic surfaces,
15 of which isolated 
while one occurs in a one-dimensional continuous family (varying with the j-invariant).
They were classified by Miranda-Persson in \cite{MP}.
We will give all the relevant equations and properties in Table \ref{T0} below,
but in a different form geared towards our purposes.
%
 Namely the extended Weierstrass forms 
 are meant to make visible 
 the singular fibers as much as possible (often obtained from Tate's (or Kubert's) normal form for elliptic curves
 with a rational point of given order, compare also Beauville's representations \cite{Beauville}).
 In particular, this will ease the process of
 endowing the quadratic twists with sections
 meeting the reducible fibers in a prescribed way
 (following Tate's algorithm \cite{Tate}).

 \begin{table}[ht!]
 $$
 \begin{array}{|cccc|}
 \hline
 \text{notation} & \text{Weierstrass eqn. } y^2 = f(x) & \text{sing. fibers} & \MW\\
 \hline
 X_{9111} & x^3+(tx+1)^2/4
 & I_9/\infty, I_1/t^3-27
 & \Z/3\Z\\
 X_{8211} & x(x^2+(t^2+2)x+1) & I_8/\infty, I_2/0, I_1/t^2+4 & \Z/4\Z\\
 X_{6321} & x(x^2+(t^2/4+t-2)x+1-t) & 
 I_6/\infty, I_3/0, I_2/1, I_1/-8 &
 \Z/6\Z\\
 X_{5511} & x^3-tx^2+((1-t)x-t)^2/4 & I_5/0,\infty, I_1/t^2-11t-1 & \Z/5\Z\\
 X_{4422} & x(x-1)(x-t^2) & I_4/0,\infty, I_2/\pm 1 & \Z/4\Z\times\Z/2\Z\\
 X_{3333} &
 x^3 + (3tx+t^3-1)^2/4
 & I_3/\infty, t^3-1 & (\Z/3\Z)^2\\
  X_{321} & x(x^2+x+t) & III^*/\infty, I_2/0, I_1/\frac14 & \Z/2\Z\\
 X_{222} & x(x-1)(x-t) & I_2^*/\infty, I_2/0,1 & (\Z/2\Z)^2\\
 X_{141} & x(x^2+t(t+2)x+t^2) & I_4/\infty, I_1^*/0, I_1/-4 & \Z/4\Z\\
 X_{\lambda} & x(x^2+tx+\lambda t^2) & I_0^*/0,\infty & (\Z/2\Z)^2\\
 \hline
 \end{array}
 $$
 \caption{Extremal rational elliptic surfaces}
 \label{T0}
 \end{table}

 Note that the
 $2$-torsion sections are always located at zeros of $f(x)$;
 here we made sure to always normalize one of them, if any over $K(t)$,  to $(0,0)$.
 
 \begin{Remark}
 \label{rem}
 The attentive reader will notice that the isotrivial extremal rational elliptic surface
 $X_{33}$
 (in the notation of \cite{MP})
 does not appear in the above table,
 although it is not distinguished from $X_{321}$
 by the root lattices associated to the singular fibers
 (see the two cases  appearing in \ref{ss:A_9}).
 The reason for its absence 
 lies in the fact that the quadratic base changes of the latter have indeed two moduli,
 while the former only yield one parameter (since there is still a scaling in the base variable),
 and in fact they form a subfamily of the
 moduli theoretic closure of the latter (cf. the analogous discussion for $X_{431}$ and $X_{44}$ in \cite[\S 3.3]{RS}).
 \end{Remark}

\section{Gluing Technique}
\label{s:glue}

We are now in the position to state the key result for the construction 
of Enriques surfaces with maximal root type.

\begin{Proposition}
\label{prop:base}
Any Enriques surface supporting a root type of rank $9$
arises from an extremal rational elliptic surface 
by the base change type construction.
\end{Proposition}

The proof of Proposition \ref{prop:base} will cover the remainder of this section.
In the end, we will be rather sketchy and provide only the key ingredients,
but to give an idea of the concepts and obstacles involved 
we will work out a few instructive cases in full detail.
The key idea is as follows:
using Proposition \ref{prop:E}
exhibit an elliptic fibration on the Enriques surface $Y$ 
whose singular fibers support root lattices of total rank $8$.
Since $\Jac(Y)$ shares the  singular fibers with $Y$,
Criterion \ref{crit} shows
that $\Jac(Y)$ is an extremal rational elliptic surface
(which we often determine explicitly for later use in the proof of Theorem \ref{thm}).
Then it remains to equip the given fibration on $Y$ with a smooth rational bisection
and apply Lemma \ref{lem:bct}.

\subsection{Proof of Proposition \ref{prop:base} for $\boldsymbol{R=A_9}$}
\label{ss:A_9}

Let $Y$ be an Enriques surface supporting the root type $A_9$
(such as the ones from Example \ref{ex:1-2}).
Recall from Example \ref{ex:1} that we have an orthogonal sum
\[
A_9 + \Z h \subset \Num(Y), \;\;\; h^2=10.
\]
We shall apply Proposition \ref{prop:E} to the isotropic vector
\[
E=a_2^\vee + 2h/5
\]
(which we obtain from the gluing isomorphism \eqref{eq:q}).
It gives rise to an elliptic fibration on $Y$.
Here $a_2$ induces a bisection and $E$ (or $-E$)  a half-pencil 
since $a_2.E=1$.
Perpendicular to $E$, we find the sublattices
\[
A_1=\Z a_1 \;\;\; \text{ and } \;\;\; A_7 = \langle a_3,\hdots,a_9\rangle
\]
which thus embed into the singular fibers.
Hence Criterion \ref{crit} shows that $\Jac(Y)$ is extremal
(a priori it could be either of $X_{8211}, X_{321}, X_{33}, X_{211}, X_{22}$,
but we will narrow these down to the first alternative in the next section);
then the claim of Proposition \ref{prop:base} follows from Lemma \ref{lem:bct}. \qed

 \subsection{Proof of Proposition \ref{prop:base} for $\boldsymbol{R=A_8+A_1}$}
 \label{ss:A_8+A_1}
 
 Let $Y$ be an Enriques surface of root type $R=A_8+A_1$.
 As indicated in Example \ref{ex:2-2}, 
 we have to distinguish two cases,
 depending on the primitive closure $R'$ of $R$ in $\Num(Y)$.

 \subsubsection{$R'=R$}
 \label{sss:A_8+A_1}
 
  If $R$ is primitive in $\Num(Y)$,
  then it is part of the index $18$ sublattice
  \[
  R + \Z h \subset \Num(Y),\;\;\; h^2=18.
  \]
 Consider the isotropic vector $E=a_1^\vee+2h/9$
 where $a_1^\vee \in A_8^\vee$.
 Then $E^2=0$,
 and since $a_1.E=1$,
 we derive half-pencil and smooth rational bisection.
 Perpendicular to $E$, we find the sublattice $A_7+A_1$,
 so again $\Jac(Y)$ is extremal by Criterion \ref{crit} (a priori with the same alternatives as in \ref{ss:A_9}),
 and Proposition \ref{prop:base} follows. \qed

\subsubsection{$R'=E_8+A_1$}
\label{sss:E_8+A_1'}

In this case,
Proposition \ref{prop:base} follows immediately from  Example \ref{ex:2-3}.
Note that $\Jac(Y)$ is determined by the reducible fiber of type $\tilde A_8$ to be $X_{9111}$.
\qed

\begin{Remark}
Of course, the gluing is by no means unique.
In fact, for both types of $A_8+A_1$,
we will use a different isotropic vector in the sequel
in order to ease computations
in working out the families of Enriques surfaces supporting them, see \ref{ss2:A_8+A_1}.
\end{Remark}

\subsection{Proof of Proposition \ref{prop:base} for all other root types}
\label{ss:glue}

For all these and the remaining root types $R$ from Theorem \ref{thm:31}
on some putative Enriques surface $Y$ 
the following table collects the following data: the primitive closure 
\[
R'\subset U+E_8 = \Num(Y),
\]
the isotropic vector $E$, given by a single dual vector of one of the summands of $R$
and a fraction of the positive vector $h$ of square $h^2=|\disc R'|$,
and the  intersection
\[
R_0=R\cap E^\perp,
\]
in each case a root lattice of rank $8$.
Together these information suffice to deduce from Criterion \ref{crit}
that $\Jac(|\pm2\sigma(E)|)$ is an extremal rational elliptic surface
(and read off the few possibilities for $\Jac(|\pm2\sigma(E)|)$).

\begin{table}
$$
\begin{array}{|cccc|}
\hline
R & R' & E & R_0\\
\hline
A_9 & A_9 & a_2^\vee + 2h/5 & A_7+A_1\\
A_8+A_1 & A_8+A_1 & (a_1^\vee,0)+2h/9 & A_7+A_1\\
& E_8+A_1 & (0,a_1^\vee) + h/2 & A_8\\
A_7 + A_2 & E_7 + A_2 & (a_2^\vee,0) + h/2 & A_5+A_2+A_1\\
A_7 + 2A_1 & E_8 + A_1 & (a_4^\vee,0,0) + h & 2A_3+2A_1\\
A_6+A_2+A_1 & A_6+A_2+A_1 & (a_1^\vee,0,0)+h/7 & A_5+A_2+A_1\\
A_5 + A_4 & A_5+A_4 & (0,a_2^\vee)+h/5 & A_5 + A_2 + A_1\\
A_5+A_3+A_1 & E_6+A_3 & (a_2^\vee,0,0)+h/3 & 2A_3+2A_1\\
A_5+2A_2 & E_7+A_2 & (0,0,a_1^\vee)+ h/3 & A_5+A_2+A_1\\
A_5+A_2+2A_1 & E_8+A_1 &(0,0,0,a_1^\vee)+h/2 & A_5+A_2+A_1\\
2A_4+A_1 & E_8 + A_1 & (0,0,a_1^\vee) + h/2 & 2A_4\\
A_4+A_3+2A_1 & A_4+D_5 & (a_1^\vee,0,0,0)+h/5 & 2A_3+2A_1\\
3A_3 & D_9 & (a_2^\vee,0,0) + h/2 & 2A_3+2A_1\\
2A_3+A_2+A_1 & E_7 + A_2 & (0,0,a_1^\vee,0) + h/3 & 2A_3+2A_1\\
2A_3+3A_1 & E_8+A_1 & (0,0,0,0,a_1^\vee)+h/2 & 2A_3+2A_1\\
A_3 + 3A_2 & A_3 + E_6 & (a_1^\vee,0,0,0) + h/4 & 4A_2 \\
D_9 & D_9 & d_9^\vee + 3h/4 & A_8\\ 
D_8+A_1 & E_8+A_1 & d_8^\vee + h & A_7+A_1\\
D_7+2A_1 & D_9 & (d_1^\vee,0,0)+h/2 & D_6+2A_1\\
D_6+A_3 & D_9 & (0,a_2^\vee)+h/2 & D_6+2A_1\\
D_6+A_2+A_1 & E_7 + A_2 & (0,a_1^\vee,0) + h/3 & D_6+2A_1\\
D_6+3A_1 & E_8+A_1 & (0,0,0,a_1^\vee) + h/2 & D_6+2A_1\\
D_5+A_4 & D_5+A_4 & (d_5^\vee,0) + h/4 & 2A_4\\
D_5+A_3+A_1 & E_8+A_1 & (0,0,a_1^\vee) + h/2 & D_5 + A_3\\
D_5+D_4 & D_9 & (0,d_1^\vee)+h/2 & D_5+A_3 \\
D_4+A_3+2A_1 & D_9 & (d_1^\vee,0,0,0)+h/2 & 2A_3+2A_1\\
2D_4+A_1 & E_8+A_1 & (0,0,a_1^\vee)+h/2 & 2D_4\\
E_8+A_1 & E_8+A_1 & (0,a_1^\vee) + h/2 & E_8\\
E_7+A_2 & E_7+A_2 & (e_7^\vee,0) + h/2 & E_6+A_2\\
E_7 + 2A_1 & E_8 + A_1 & (0,0,a_1^\vee) + h/2 & E_7 + A_1 \\
E_6+A_3 & E_6+A_3 & (e_6^\vee,0)+h/3 & D_5+A_3\\
E_6+A_2+A_1 & E_8+A_1 & (0,0,a_1^\vee)+h/2 & E_6+A_2\\
\hline
\end{array}
$$
\caption{Isotropic vectors and root lattices for the proof of Proposition \ref{prop:base}}
\label{T1}
\end{table}

To complete the proof of Proposition \ref{prop:base},
it suffices to argue with the $(-2)$ curve $C\subset Y$
whose dual vector is used for the gluing.
By construction, we have $C.E=1$ in either case,
so $E$ induces a half-pencil and $C$  a smooth rational bisection for the fibration by Proposition \ref{prop:E}.
Hence Lemma \ref{lem:bct} implies that $Y$ arises from an Enriques involution of base change type
originating from the extremal rational elliptic surface $\Jac(|\pm2\sigma(E)|)$. \qed

\section{Fiber analysis}
\label{s:fiber}

We shall now provide the only missing puzzle piece
needed to translate the  data found so far into explicit equations.
In particular, this analysis will be crucial for the determination of all
the moduli components for a given root type $R$ from Theorem \ref{thm:31}
(as opposed to a sole existence statement).

To this end, we start from the data provided in Table \ref{T1} 
which endow an Enriques surface $Y$ with given root type $R$ with an elliptic fibration $|2E'|$ by Proposition \ref{prop:E}.
We then analyse
how the (anti-)effective $(-2)$-divisors originating from $R_0$
by way of the composition of reflections $\sigma$ in the proof of Proposition \ref{prop:E}
may embed into the singular fibers of the elliptic fibration.
Recall from Remark \ref{rem:curves?} that in general we cannot expect these divisors to be single curve classes.
Presently, however, this luckily happens to be the case (much like in Example \ref{ex:2-3}):

\begin{Proposition}
\label{prop:curves}
In the above set-up, all curves representing $R_0$ are mapped to fiber components of the elliptic fibration $|2E'|$
by a suitable composition of reflections. 
In particular, their intersection pattern with the smooth rational bisection $B'$ can be read off directly from $R$,
only depending on the configuration of singular fibers and their multiplicities.
\end{Proposition}

As before, we exhibit the detailed proof of 
Proposition \ref{prop:curves} only for two root types
and omit the tedious rest for space reasons.
Without further reference, we will use standard facts about root lattices and their embeddings,
especially uniqueness properties up to the action of the Weyl group (see for instance \cite{Bourbaki}).

 \subsection{Proof of Proposition \ref{prop:curves} for $R=A_9$}
 \label{ss:pf1}
 
 Recall from \ref{ss:A_9}
 that there is a half-pencil $E'$
with smooth rational bisection $B'=\sigma(a_2)$
and perpendicular $R_0=A_7+A_1$.
Without further information, these lattices embed into three configurations of reducible fibers
of an Enriques surface (or an (extremal) rational elliptic surface):
\begin{eqnarray}
\label{eq:3}
\tilde A_7 + \tilde A_1, \;\; \tilde E_7 + \tilde A_1, \;\; \tilde E_8
\end{eqnarray}
(which lead to the extremal rational elliptic surfaces listed in \ref{ss:A_9}).
The last alternative, however, can be excluded right away since it would only have one simple fiber component
which thus meets $B'$ with multiplicity two. Hence there can be no divisors supported on the fibers 
(such as $\sigma(a_1), \sigma(a_3)$) intersecting $B'$ with multiplicity one, contradiction.
 
 For the second case from \eqref{eq:3}, i.e.~$\sigma(A_7)\hookrightarrow \tilde E_7$,
 there is a similar argument: since $\sigma(a_3).B'=1$, the bisection $B'$ has to intersect
 both simple fiber components transversally 
 while $\sigma(a_3)$ only involves one of them and $\sigma(a_i)$ neither for $i=4,\hdots,9$.
That is, we can omit one fiber component, and $\sigma(A_7)\hookrightarrow E_7$.
But then there is a divisible class in $\sigma(A_7)$ and thus in $A_7$;
however, $A_7$, being an orthogonal summand of the primitive lattice $R_0=R\cap E^\perp$,
 is itself primitive in $\Num(Y)$, contradiction.
 
 It remains to investigate the first alternative from \eqref{eq:3}.
 We start by analysing the embedding $\sigma(A_1)\hookrightarrow \tilde A_1$.
 Since $\sigma(a_1).B'=1$, the bisection $B'$ meets either one or both components transversally.
 In the second case, $\sigma(a_1)$ can only be one of the components as claimed
 while in the first case, which corresponds to a multiple $I_2$ fiber,
 we could also have $\sigma(a_1)=\Theta+\Theta'$,
 but then a reflection in the component not met by $B'$ yields the claim.
 
 We now turn to the embedding $\sigma(A_7)\hookrightarrow \tilde A_7$.
 First note that the $I_8$ fiber cannot ramify.
 Otherwise the K3 cover $X$ would attain an $I_{16}$ fiber
 which would be met by the section $P$ (mapping to $B'$)
 in the component opposite to the identity component
 (in a two-torsion point of the singular fiber
 which is off the identity component, compare the logarithmic transformation).
 We continue to argue with the height from the theory of Mordell-Weil lattices \cite{ShMW}, 
 applied to $X$.
 Since $P$ meets the two $I_2$ fibers nontrivially (or in case of ramification, the $I_4$ fiber),
 as we have seen above,
the height returns
 \begin{eqnarray}
 \label{eq:ht}
 h(P) = 4-\frac{8\cdot 8}{16} - 1 = -1 <0,
 \end{eqnarray}
but this is impossible. Hence the $I_8$ fiber is reduced and met by $B'$ in two different components, say $\Theta_0, \Theta_j (0<j<8)$
under a cyclic numbering,
for otherwise $\sigma(a_3).B'=1$ would be impossible.
Without loss of generality, we may assume that $\sigma(a_3)$ does not contain $\Theta_0$ in its support,
and neither do the other $\sigma(a_j)$
which are also off $\Theta_j$, since they are orthogonal to $B'$.
Hence there is an embedding
\[
A_6 =\langle\sigma(a_4),\hdots,\sigma(a_9)\rangle \hookrightarrow A_{j-1} + A_{7-j} = \langle \Theta_1,\hdots,\Theta_{j-1}\rangle
+ \langle \Theta_{j+1},\hdots,\Theta_7\rangle.
\]
This is only possible for $j=1$ or $j=7$, i.e.~$B'$ meets two adjacent fiber components.
After changing the orientation, if necessary, we may assume $j=1$
and apply successive reflections in $\Theta_2,\hdots, \Theta_7$
to ensure that
\[
\sigma(a_i) = \Theta_{i-2} \;\;\; (i=4\hdots,9).
\]
But then the given intersection numbers imply that $\sigma(a_3)=\Theta_1$, and the claims of Proposition \ref{prop:curves} follow.
\qed

Note that the above findings are consistent with Figure \ref{fig:A_9} in Example \ref{ex:1-2}.
 
 \begin{Remark}
 For the base change type construction,
 we could have terminated our calculations where $B'$ was found to met adjacent fiber components
 because this is all the information needed to compute the sections $P'$ and thus $P$.
 \end{Remark}

 \subsection{Proof of Proposition \ref{prop:curves} for $R=A_8+A_1$}
 \label{ss:A8+A1}
 
 In the imprimitive case $R'=E_8+A_1$, we have already seen the claims of Proposition \ref{prop:curves} in Example \ref{ex:2-3}.
 Thus it remains to study the primitive case, with $R_0=A_7+A_1$
 and $\sigma(a_2).B'=1$.
 For the fiber configurations, we have the same candidates as \eqref{eq:3};
in fact, the second and third alternative are ruled out exactly as in \ref{ss:pf1}. 
The same arguments also apply to the first alternative from \eqref{eq:3} if the $I_8$ fiber is unramified.
Contrary to \ref{ss:pf1}, however, the first alternative allows for the $I_8$ fiber to be multiplicative as well
(so that $P$ has height zero and thus order two).
In that case, the bisection $B'$ meets a single fiber component, say $\Theta_0$
which thus is part of the support of $\sigma(a_2)$, but not for $\sigma(a_i)$ for $i=3,\hdots,8$.
Ignoring $\Theta_0$, we obtain an embedding
\[
A_6 \cong \langle\sigma(a_3),\hdots,\sigma(a_8)\rangle \hookrightarrow A_7.
\]
After some further reflections, we can thus assume that 
\[
\sigma(a_i) = \Theta_{i-2} \;\;\; (i=3,\hdots,8).
\]
 Then the intersection numbers predict that $\sigma(a_2)=\Theta_0$,
 and Proposition \ref{prop:curves} is proven in all instances of $R=A_8+A_1$.
 \qed

 \subsection{Important observation}
 
 For immediate use in the moduli component analysis,
 we record the following by-product of the proof of Proposition \ref{prop:curves}:
 
 \begin{Corollary}
 \label{cor:by}
 In the above set-up,
 the root type $R_0$ always is supported on a unique configuration of singular fibers
 (in fact given by the $\tilde R_v$ for the orthogonal summands of $R_0$)
 unless
 there is an additional fiber configuration to be found in Table \ref{Tx}.
 \end{Corollary}

\begin{table}[ht!]
$$
\begin{array}{|ccc|}
\hline
R & R'  & \text{fibers}\\
\hline
A_7+2A_1 
& 2A_3+2A_1 & \tilde D_6+2\tilde A_1\\
&& \tilde E_7 + \tilde A_1\\

A_5+A_3+A_1 
& 2A_3+2A_1 & \tilde D_5+\tilde A_3\\

3A_3 
& 2A_3+2A_1 & \tilde D_5+\tilde A_3\\

D_6+A_3 
& D_6+2A_1 & \tilde D_9\\
\hline
\end{array}
$$
\caption{Additional fiber configurations}
\label{Tx}
\end{table}

 We emphasize that the corollary does not take multiple fibers into account yet.
 These will enter the picture in the next sextion.

 \section{One-dimensional families}
 \label{s:fam}
 
 It is about time to take advantage of the precise (though abstract) information about the putative Enriques surfaces
 obtained from the maximal root types in the previous sections
 in order to produce explicit parametrizations and thereby 
prove Theorem \ref{thm}.
 The overall strategy should be clear:
 given an Enriques surface $Y$ supporting a root lattice from Theorem \ref{thm:31},
 Proposition \ref{prop:base} tells us that $Y$ arises from an extremal rational elliptic surface $S$
 through the base change type construction of Enriques involutions.
 Indeed, Table \ref{T1} mostly  predicts  $S$,
 with a few exceptions given in Corollary \ref{cor:by} and Table \ref{Tx}.
 Moreover, Proposition \ref{prop:curves} paves the way
 to exactly pinpoint the bisection $B'$ on $Y$
 and how it intersects the reducible singular fibers.
  This information easily translates into the section $P$ on the K3 cover $X$,
 and the underlying section $P'$ on some quadratic twist $S'$ of $S$.
 Recall from Section \ref{s:tech} how the shape of $P'$ is predicted:
 as a polynomial $x(P')$ of degree $2$ (see \eqref{eq:x(P')}).
 This puts us in the  position
 to aim directly for a parametrization of all quadratic twists $S'$ admitting a section $P'$ of the required shape.

\subsection{Explicit computations}
\label{ss:explicit}

It is a rather exceptional instance that we can parametrize all families of Enriques surfaces explicitly.
Before going into a few details, we give a brief account of the obstacles against this approach.

Roughly, there are three cases when going for the explicit parametrizations,
starting from an extremal rational elliptic surface $S$.
Recall that we want to implement the quadratic twists $S'$ of $S$
in such a way that they attain an integral section $P'$ of prescribed shape.

\subsubsection{}
\label{sss:2f}
If $P'$ meets two different singular fibers non-trivially 
(in addition to $I_0^*$ fibers originating from the twist, such as in  Example \ref{ex:1-2}),
 or one fiber at a fiber component not adjacent to the identity component,
 then this provides two conditions for $x(P')$ which is thus determined up to one parameter;
hence we directly derive the desired family of Enriques surfaces.

 \subsubsection{}
 \label{sss:1f}
 If $P'$ meets  one  singular fiber (in addition to $I_0^*$ fibers originating from the twist) 
 in a component adjacent to the identity component,
then this gives one condition on $x(P')$ and thus leaves two parameters.
Substituting into the RHS of the extended Weierstrass form of $S$ from Table \ref{T0}
gives a square factor and  a degree $4$ polynomial $g$ in $t$.
Then we simply solve for the discriminant of $g$ to vanish 
while preserving the required shape of $P'$
and ensuring that $q$ does not degenerate.

\subsubsection{}
\label{sss:0f}
If $P'$ meets no singular fibers other than the two $I_0^*$ fibers originating from the twist
non-trivially,
then this gives no new conditions on $x(P')$
and there are three parameters to master.
In general, this complicates computations substantially,
but one can still handle them. 

\subsubsection{Effect of two-torsion}
\label{sss:2t}

Even in the last case,
the situation changes drastically in the presence of two-torsion
in the Mordell-Weil group.
Normalising the twisted Weierstrass form of $S'$ as in Table \ref{T0} to
\[
S':\;\;\; qy^2 = x(x^2+a_2x+a_4),\;\;\; a_i\in K[t],\;\; \deg a_i\leq i,
\]
we find that $x(P')$ and $a_4$ are coprime,
since otherwise $P'$ would intersect the corresponding singular fiber non-trivially.
In order to be able to solve for $y(P')$,
we dedude that 
\begin{itemize}
\item 
either $x(P')$ is a perfect square 
(thus reducing the number of parameters to two,
so that one can solve as in \ref{sss:1f})
\item
or
$x(P')=\alpha q$ for a scalar $\alpha\in K$.
Here one can solve directly for the remaining factor on the RHS upon substitution,
 $x(P')^2+a_2 x(P')+a_4$
to give a perfect square.
\end{itemize}
Of course, these ideas also apply to simplify our considerations if $P'$ meets some singular fibers non-trivially,
but then we have to pay a closer look into the precise conditions resulting since $x(P')$ and $a_4$ need not
be coprime anymore (see \ref{ss:prim} for an exemplary case).

\subsection{Round-up for root type $R=A_9$}
\label{ss:A_9'}

For this root type and the data from Table \ref{T1},
Proposition \ref{prop:curves} and the details of its proof in \ref{ss:pf1}
predict that $\Jac(Y) = X_{8211}$ 
and $P'$ meets both reducible fibers, of Kodaira type $I_8$ and $I_2$ non-trivially.
Hence \ref{sss:2f} kicks in to yield exactly the family of Enriques surfaces from Example \ref{ex:1-2}
supporting $R$.
Recall from \ref{ss:pf1} that the $I_8$ fiber is unramified
while a priori the $I_2$ fiber may be multiple.
Indeed, we read off from the given parametrization that this happens exactly at $\mu=\pm 1$. 

In conclusion, we have verified that
any Enriques surface supporting the root type $R=A_9$ lies in the given family.

\subsection{A first look at the moduli components}

A similar picture persists for all other root types from Table \ref{T1}.
We omit the details of the computations for space reasons,
but we will soon list all the families of Enriques surfaces explicitly (Table \ref{T2}).
However, we will sometimes switch to a different model of the families
which is more convenient (in particular, less complicated, already in terms of the algebraic expressions involved).
Since this thread of thought is somewhat orthogonal to the present one,
we postpone listing the explicit equations until Section \ref{s:explicit}
and pay a first closer look to the moduli components involved.

The moduli components are directly obtained from the explicit parametrization approach.
Usually one expects only one irreducible component
(such as for $R=A_9$ in \ref{ss:A_9'}),
but there are three possible reasons for additional components:
\begin{enumerate}[(i)]
\item
the data from Table \ref{T1} may allow for different fiber configurations (captured in Table \ref{Tx});
\item
given a fiber configuration, different fiber multiplicities may admit distinct families;
\item
the parametrization of a given type with fixed fiber multiplicities may not be unique up to symmetries.
\end{enumerate}
We will see in Proposition \ref{prop:unique} that the third case essentially does not occur.
Meanwhile the second alternative may be a bit counter-intuitive,
since ramification at a  singular fiber is a codimension one-condition in the moduli of quadratic base changes,
but there is a surprising twist.
Namely ramification at suitable fibers may allow for the section $P$ to be two-torsion
so that it may already be present on $S$ and $S'$ at no extra cost in terms of parameters.
We illustrate this with an investigation of  the following root type.

\subsection{Primitive root type $R=A_8+A_1$}
\label{ss:prim}

Recall from \ref{ss:A8+A1} that the data from Table \ref{T1} for the primitive case of $R=A_8+A_1$
lead to $\Jac(Y)=X_{8211}$.
More importantly, they
allow for the $I_8$-fiber to be ramified or unramified.
We discuss both situations.

If the $I_8$ fiber is unramified, then by \ref{ss:A8+A1} 
the bisection $B'$ meets it in two adjacent fiber components
while meeting the $I_2$ fiber in a single component as depicted below
(where the dashed line indicates that the $I_2$ fiber may ramify):

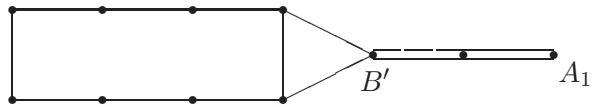
\begin{figure}[ht!]
\setlength{\unitlength}{.6mm}
\begin{picture}(120,22)(0,0)
%
\multiput(0,0)(20,0){4}{\circle*{1.5}}
\put(0,0){\line(1,0){60}}


\multiput(0,20)(20,0){4}{\circle*{1.5}}
\put(0,20){\line(1,0){60}}



\put(0,0){\line(0,1){20}}
\put(60,0){\line(0,1){20}}

\put(77,2){$B'$}


\put(60,20){\line(2,-1){20}}
\put(60,0){\line(2,1){20}}

\put(80,11){\line(1,0){6}}
\put(87,11){\line(1,0){6}}
\put(94,11){\line(1,0){6}}

\put(80,9){\line(1,0){40}}
\put(100,11){\line(1,0){20}}

\multiput(80,10)(20,0){3}{\circle*{1.5}}
\put(121,4){$A_1$}

%
%

\end{picture}
\caption{$A_8+A_1$-configuration with an unramified $I_8$ fiber}
\label{fig:ram8}
\end{figure}

Thus \ref{sss:1f} applies (as a degenerate case of \ref{sss:2t}),
and
\[
x(P') = \mu t + \lambda, \;\; \; \mu\neq 0, \lambda \neq -1
\]
(since otherwise $P'$ would meet different fiber components).
The resulting discriminant vanishes exactly when
\[
\mu = -\dfrac{2s^3}{s^2+1}, \;\; \lambda = \dfrac{s^4-3s^2}{s^2+1}
\]
for some parameter $s\neq 0,1$, so we obtain the desired irreducible family of Enriques surfaces.

On the other hand,
we already noted before that ramification at the $I_8$ fiber causes $P$ to be two-torsion;
in fact, $P=(0,0)$, and we can let the second ramification point vary (outside the other singular fibers,
since otherwise the involution would attain fixed points).
That is,
\[
q = t-\lambda, \;\;\; \lambda \neq 0, \pm 2\sqrt{-1},
\]
and we can depict the resulting smooth rational curves on the one-dimensional family of Enriques surface as follows:

\begin{figure}[ht!]
\setlength{\unitlength}{.6mm}
\begin{picture}(120,22)(0,0)
%
\multiput(0,0)(20,0){7}{\circle*{1.5}}
\put(0,0){\line(1,0){80}}


\multiput(0,20)(20,0){4}{\circle*{1.5}}
\put(0,20){\line(1,0){60}}



\put(0,0){\line(0,1){20}}
\put(60,0){\line(0,1){20}}


%

\put(77,4){$B'$}

\put(80,-1){\line(1,0){40}}
\put(80,1){\line(1,0){40}}

\put(117,4){$A_1$}

%
%

\end{picture}
\caption{$A_8+A_1$-configuration with a ramified $I_8$ fiber}
\label{fig:mult8}
\end{figure}
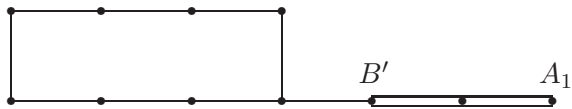

\begin{Remark}
\label{rem:over}
It is important to note that the $(-2)$-curves in Figure \ref{fig:mult8} also support the root types
$A_7+2A_1, D_8+A_1, E_8+A_1$.
\end{Remark}

\subsection{Rough component analysis}

Let us return to the components of the moduli spaces of the root types from Theorem \ref{thm:31}.
Using our explicit approach towards parametrizing the Enriques surfaces with maximal root types
(which will ultimately lead to Table \ref{T2}),
we find the following two fundamental results.
The first is a sheer existence result:

\begin{Proposition}
\label{prop:ex}
Any root type from Theorem \ref{thm:31}
is supported on a one-dimensional family of Enriques surfaces
with reduced configuration of singular fibers given by the data in Table \ref{T1}.
\end{Proposition}

\begin{Remark}
The same statement holds for the additional fiber configurations from Table \ref{Tx},
except for those for $R=A_7+2A_1$ (see the next proposition).
\end{Remark}

\begin{Proposition}
\label{prop:unique}
For any fiber configuration originating from Table \ref{T1},
i.e.~the singular fibers are either given as $\tilde R_v$ for the orthogonal summands of $R_0$
or they are in Table \ref{Tx},
there is, up to isomorphism,
 a unique one-dimensional family of Enriques surfaces supporting the root type $R$ on reduced fibers,
except in the following cases:
\begin{itemize}
\item
$R=D_6+3A_1$ where there are two families (see \ref{sss:D_6+3A_1});
\item
$R=A_7+2A_1$ with fibers $\tilde D_6+2\tilde A_1$ or $\tilde E_7+\tilde A_1$ where there
automatically is a multiple $I_2$ fiber (see \ref{ss:721});
\item
there are two families, one without ramified fibers and one with a multiple fiber as listed in Table \ref{Ty}.
\end{itemize}
\end{Proposition}

\begin{table}[ht!]
$$
\begin{array}{|cccc|}
\hline
R & R' &  R_0 & \text{possible fiber configurations}\\
\hline
A_8+A_1 & A_8+A_1 & A_7+A_1 & 
2I_8, I_2, I_1, I_1\\

A_7 + 2A_1 & E_8 + A_1 & 2A_3+2A_1 & I_2^*, 2I_2, I_2\\
A_5+2A_2 & E_7+A_2 &  A_5+A_2+A_1 & 
2I_6, I_3, I_2, I_1\\

3A_3 & D_9 &  2A_3+2A_1 & 
2I_4, I_4, I_2, I_2\\
&&& I_1^*, 2I_4, I_1\\

D_8+A_1 & E_8+A_1  & 
A_7+A_1 & 
2I_8, I_2, I_1, I_1\\

D_4+A_3+2A_1 & D_9 & 2A_3+2A_1 &
2I_4, I_4, I_2, I_2\\

\hline
\end{array}
$$
\caption{Fiber configurations admitting several moduli components}
\label{Ty}
\end{table}

For most root types, we can thus already answer the problem of the number
of moduli components:

\begin{Corollary}
\label{cor:unique}
If $R\neq D_6+3A_1$ is not listed in Table \ref{Tx} or \ref{Ty},
then its moduli space is irreducible.
\end{Corollary}

Conversely, we have an easy observation concerning
certain overlattices being supported on a given family of Enriques surfaces (compare Remark \ref{rem:over}).

\begin{Observation}
\label{obs}
If $R$ is supported on an Enriques surface $Y$
with fiber configuration from Table \ref{Tx} or \ref{Ty},
then $Y$ supports an integral overlattice $\hat R$ of $R$ inside $R'$.
\end{Observation}

This observation has a big impact on the rest of the paper.
For one thing, it already indicates that the moduli components should be distinct
(which one can verify indeed, see Section \ref{s:number}).
Secondly it eases the presentation of all the one-dimensional families of Enriques surfaces to follow.
Namely the extra moduli components for $R$ will  be listed (in Table \ref{T2}), without further specification, 
under the respective overlattices $\hat R$.

 \section{Explicit equations}
 \label{s:explicit}

 This section presents the results of our explicit parametrizations
 (which we used to prove Propositions \ref{prop:ex} and \ref{prop:unique},
 for instance).
 We postponed their presentation because for some root types
 we chose to list a different parametrization (compared to the data from Table \ref{T1})
 which is less complicated.
 
 \begin{Theorem}
 \label{thm:any}
 Any Enriques surface supporting a maximal root type (from Theorem \ref{thm:31})
 can be found in Table \ref{T2}.
 \end{Theorem}
 
 \begin{Remark}
 Note that there is a duplicate in Table \ref{T2}
 as the root types $R=D_6+3A_1$ and $E_7+2A_1$ 
 file under the same family (see Remark \ref{rem:prim}).
 \end{Remark}
 

 The method to prove Theorem \ref{thm:any}
 has been explained in thorough detail in \ref{ss:explicit}.
 Indeed we have seen explicit parametrizations for $R=A_9$ in \ref{ss:A_9'},
 and for (primitive root type) $R=A_8+A_1$ in \ref{ss:prim}.
 After providing the equations in \ref{ss:eq},
 we will include explanations for all those root types
 where the families in Table \ref{T2}
 does not originate directly from the data in Table \ref{T1}
 (and explain the reason).

\subsection{Explicit equations}
\label{ss:T3}
\label{ss:eq}

 The following table supplements Theorem \ref{thm:any}
 by collecting all Enriques surfaces supporting a root type from Theorem \ref{thm:31}.
 As it happens, they  all appear in one-dimensional families as stated in Theorem \ref{thm}.
 The base is always $\PP^1$,
 with varying parameter $\mu, \lambda, s$ or $u$.
The Enriques surfaces $Y$ are given in terms of
\begin{itemize}
\item
an extremal rational elliptic surface $S$ which is
the jacobian $\Jac(Y)$ of an elliptic fibration on $Y$
(cf.~Proposition \ref{prop:base}),
\item
its quadratic twist $S'$, given by a quadratic polynomial $q\in K(\PP^1)[t]$,
\item
a section $P'$ of $S'$ (given by $x(P')$)
which induces the Enriques involution of base change type on the K3 cover $X$ of $Y$.
\end{itemize}
Unless stated in the subsequent paragraphs, the respective elliptic fibration is obtained
in the canonical  way with reduced fiber configuration from the data in Table \ref{T1}
(as stated in Proposition \ref{prop:unique}).

 The reader should be advised that there always is finite set of places of $\PP^1$
 where the parameter ceases to give an Enriques surface due to degenerations,
notably ramification points coming together or the involution attaining fixed points.
 We do not list these places explicitly, as they are easily calculated.

 \begin{table}[ht!]
 $$
 \begin{array}{|lccc|}
 \hline
 \text{root type} & {\Jac(Y)} & \text{section:} \, x(P') & \text{quadr.~twist:} \, q\\
 \hline
A_9 & X_{8211} & \mu t -1 & (\mu t-1) (\mu t +\mu^2-1)\\
A_8 + A_1 & X_{6321} & \mu t + 1 & (\mu t+1)(\mu t + (2\mu+1)^2)
\\
A_7+A_2 & X_{6321} & \mu(t-1)
&
\mu t^2+4\mu (\mu+1) t-4(\mu+1)^2
\\
A_7+2A_1 & X_{4422} & \mu t & (\mu t-1)(t-\mu)\\
A_6+A_2+A_1 & X_{6321} &
-(ut-4u^2+3u-1)/(u^2(u+1))
&\\
&&\multicolumn{2}{r|}{((u^2+u)t-(u-1)^2)(ut-4u^2+3u-1)}\\
A_5+A_4 & X_{6321} & (t-1)(\mu t-1) & (\mu t-1)(\mu(4\mu+1)t-(2\mu+1)^2)\\
A_5+A_3+A_1 & X_{4422} &
\mu t+1-\mu & 
(\mu t+1-\mu)(t+1-\mu)
\\
A_5+2A_2 & X_{6321} & 27(t+u)(t-1)/((u-8)(u+1)^2) &
(t+u)(3(u+4)t-(u-2)^2)\\
A_5+A_2+2A_1 & X_{6321} & 27(t-u)^2/((u-1)(u+8)^2) &\\
&&\multicolumn{2}{r|}{
3(u+11)t^2-4(u^2+13u+4)t+4(u-4)^2}
\\
2A_4+A_1 & X_{5511} &
\multicolumn{2}{l|}{-(s^5t^2-s^3(11s^2-15s+5)t-(s-1)^5)/(5s^2-5s+1)^2}
\\
&&\multicolumn{2}{r|}
{s^2(4s-1)t^2-2(2s-1)(11s^2-11s+2)t-(4s-3)(s-1)^2}
\\
A_4+A_3+2A_1 & X_{4422} &
4\mu(t+\mu) &
(t+\mu) 
(4\mu t-4\mu^2-1)\\
3A_3 & X_{4422} & \mu t^2-\mu+1
& \mu t^2-\mu+1 \\
2A_3 + A_2 + A_1 & X_{4422} & (\mu t-\mu+1)^2 & 
(\mu t-\mu+2)(\mu t+t-\mu+1)
\\
2A_3+3A_1 & X_{4422} &
-(t^2+4\mu t+1)/(4\mu^2-1)&
t^2+4\mu t+1
\\
A_3+3A_2 
& X_{3333} & \multicolumn{2}{l|}{(t-1)((s^3+6s^2+9s+3)t+s^3+3s^2-3)/s^2(2s+3)}\\
&&\multicolumn{2}{r|}{((6s+3)(s+1)^2t^2+(8s^3+12s^2-6s-6)t+4s^3-6s+3)}\\
D_9 & X_{9111} & \mu & 
\mu^2t^2+2\mu t+4\mu^3+1
\\
D_8+A_1 & X_{8211} & \mu & \mu t^2 + (\mu+1)^2
\\
 D_7+2A_1 & X_{222} & t+\mu & (t+\mu)(t+\mu-1)
 \\
 D_6+A_3 & X_{222} & (\mu t+1-\mu)t & (\mu t+1-\mu)(\mu t+1)\\
 D_6+A_2+A_1 & X_{222} & -4\lambda t(\lambda t-1)
 & (\lambda t-1)(4\lambda^2 t-4\lambda+1)
 \\
 D_6+3A_1 & X_{222} & 0 & t(t-\lambda)\\
 D_5+A_4 & X_{5511} & \multicolumn{2}{r|}{ \mu \;\;\;\;\;\; \;\;\;\;\;\;\;\;\;\;\;\;\;\;
 (\mu+1)^2t^2-2\mu(3\mu+1)t+\mu^2(4\mu+1)}
 \\
 D_5 + A_3 + A_1 
 & X_{4422} & 1 & (t-1)(t-\lambda)
 \\
 D_5+D_4 & X_{141} & \mu & (\mu+1)t^2+2\mu t + \mu^2\\
 D_4+A_3+2A_1 & X_{4422} &\mu & t^2-\mu\\
  2D_4+A_1 & \multicolumn{2}{l}{\;X_{-4\mu(4\mu+1)} \;\;\;\;\;\;\;\;\;\;\;\;\;\;\;\;\; (t+\mu)^2} & t^2+(6\mu+1)t+\mu^2
 \\
 E_8+A_1 & X_{321} & 0 & t(t-\lambda)\\
 E_7+A_2 & X_{6321} & 0 & t-\lambda\\
 E_7+2A_1 & X_{222} & 0 & t(t-\lambda)\\
 E_6+A_3 & X_{141} & (\mu t-1)t & (\mu t-1)(\mu(\mu+1) t-1)
 \\
 E_6+A_2+A_1 & X_{6321}  & 0 & (t-1)(t-\lambda)
 \\
 \hline
 \end{array}
 $$
 \caption{Explicit families of Enriques surfaces supporting the 31 maximal root types}
 \label{T2}
 \end{table}

\subsection{Root type $R=A_8+A_1$}
\label{ss2:A_8+A_1}
\label{ss2:E_8+A_1}
\label{ss:not}

Previously
we distinguished three cases:
\begin{enumerate}
\item
\label{it1}
$R$ primitive and $\Jac(Y)=X_{8211}$ with unramified $I_8$ fiber;
\item
\label{it2}
$R$ primitive and $\Jac(Y)=X_{8211}$ with ramified $I_8$ fiber;
\item
\label{it3}
$R'=E_8+A_1$ and $\Jac(Y)=X_{9111}$.
\end{enumerate}
For each case, starting from the isotropic vector $E$ from Table \ref{T1},
one can work out the Enriques involution of base change type,
based on the extremal rational elliptic surface $X_{8211}$ resp. $X_{9111}$.
In particular, this shows that each case occurs in a one-dimensional family.

However, there is a surprising symmetry which the above approach does not detect:

\begin{Lemma}
\label{lem:iso}
The Enriques surfaces supporting the first and the third configurations form the same family.
\end{Lemma}

To prove Lemma \ref{lem:iso} it is convenient to work with a different isotropic vector than in Table \ref{T1}
which at the same time facilitates the representation of the family (especially for \eqref{it3}
which in the above form would be rather complicated).

Given an Enriques surface $Y$ supporting the root type $A_8+A_1$,
we argue with the dual vector $a_3^\vee\in A_8^\vee$ of square $(a_3^\vee)^2=-2$.
Note that as an abstract lattice $E_8=\langle A_8,a_3^\vee\rangle$, so $a_3^\vee\in R'$ in case \eqref{it3},
but not in the other two cases.
Consider the isotropic vector
\[
E = \begin{cases}
a_3^\vee + h/3, & \text { in case \eqref{it1}, \eqref{it2}},\\
a_3^\vee + h, & \text{ in case \eqref{it3}}
\end{cases}
\]
(after checking that in case \eqref{it1}, we indeed have $E\in\Num(Y)$).
Since $a_3.E=1$, $E$ induces the half-pencil of some elliptic fibration on $Y$ and $a_3$ a bisection
by Proposition \ref{prop:E}.
Perpendicular to $E$ in $R$, we find $A_1$ and inside $A_8$
\[
A_2 = \langle a_1,a_2\rangle,\;\;\; A_5 = \langle a_4,\hdots,a_8\rangle,
\]
i.e.~$R_0=A_5+A_2+A_1$.
By Criterion \ref{crit}, the rational elliptic surface $\Jac(Y)$ is extremal.
We distinguish the possible cases which a priori comprise
\begin{eqnarray}
\label{eq:cases}
X_{6321}, \; X_{431}, \; X_{44}, \; X_{321}, \; X_{33}, \; X_{211}, \; X_{22}.
\end{eqnarray}

\subsubsection{$\Jac(Y) = X_{6321}$}
\label{sss:6321}

Assume that the fibration $|2E'|$ has fibers $\tilde R_v$ for the orthogonal summands of $R_0$.
Then $\Jac(Y)=X_{6321}$, and we are led to construct a quadratic twist $S'$ with section $P'$ meeting
both $I_6$ and $I_3$ fibers in components adjacent to the identity component.
Thus we are in the situation from \ref{sss:2f} which directly leads to the family from Table \ref{T2}.
Note that the degenerate case with double $I_3$ fiber occurs at $\mu=-1/2$
while a double $I_6$ fiber cannot occur by inspection of the height pairing similar to \eqref{eq:ht}.

\subsubsection{$\Jac(Y) = X_{321}$}
\label{ss:321}

Alternatively we may embed 
\[
A_5+A_2\hookrightarrow \tilde E_7
\]
with the bisection $B'$ meeting both simple components.
On the K3 cover $X$, this leads to a section $P$ meeting both $III^*$ fibers non-trivially.

\begin{Claim}
\label{claim}
$P=(0,0)$, the two-torsion section.
\end{Claim}

\begin{proof}
Otherwise we compute the height pairing from the theory of Mordell--Weil lattices \cite{ShMW} between the two sections.
Since both are disjoint from the zero section and meet both $III^*$ fibers non-trivially, the height pairing returns
\[
0 = \langle P,(0,0)\rangle = 2 - 2\cdot \frac 32 - \hdots \leq -1,
\]
giving the required contradiction.
\end{proof}

Recall that $\sigma(A_1)$ is orthogonal to the bisection $B'$ and supported on a fiber
which presently can only be $\tilde A_1$.
But $P$ always meets a different fiber component than the zero section,
hence the above situation can only persist on the Enriques quotient $Y$
if the fiber is ramified (and thus of type $I_2$ so that $\Jac(Y)=X_{321}$ as stated).
This gives the equations listed in Table \ref{T2} under $R=E_8+A_1$ --
since this root lattice is easily verified to be supported on the resulting $(-2)$-curves,
see the following figure.

\begin{figure}[ht!]
\setlength{\unitlength}{.6mm}
\begin{picture}(120,40)(0,0)
%
\multiput(20,0)(20,0){3}{\circle*{1.5}}
\put(20,0){\line(1,0){40}}


\multiput(20,40)(20,0){3}{\circle*{1.5}}
\put(20,40){\line(1,0){40}}



\put(20,0){\line(0,1){40}}
\put(20,20){\line(-1,0){20}}

\put(0,20){\circle*{1.5}}
\put(20,20){\circle*{1.5}}

\put(60,40){\line(1,-1){20}}
\put(60,0){\line(1,1){20}}

\put(80,20){\line(1,0){20}}

\put(100,19){\line(1,0){20}}
\put(100,21){\line(1,0){20}}

\multiput(80,20)(20,0){3}{\circle*{1.5}}
\put(121,14){$A_1$}

%
%

\end{picture}
\caption{$A_8+A_1$ and $E_8+A_1$ supported on fibers of type $III^*, I_2$ and bisection}
\label{fig:321}
\end{figure}
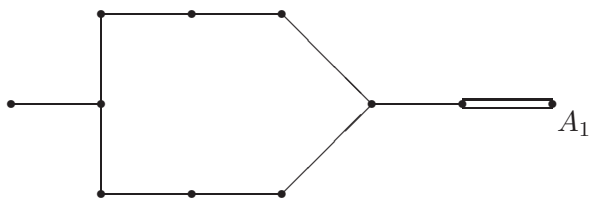

\subsubsection{Other fiber configurations}

The other cases from \eqref{eq:cases} cannot occur.
If $R_0\hookrightarrow \tilde E_8$,
then $\Theta.B'\equiv 0\mod 2$ for any fiber component,
so $\sigma(a_4).B'=1$ would not be possible.

On the other hand, if $A_5+A_1\hookrightarrow \tilde E_6$,
then for $\sigma(a_4).B'=1$ to be possible, the bisection
$B'$ has to intersect two different fiber components.
But then one of them is not contained in the support of $\sigma(A_5+A_1)$,
i.e.~$\sigma$ gives an embedding $A_5+A_1\hookrightarrow E_6$.
As we have used before, this yields an (anti-)effective two-divisible class,
such that Lemma \ref{lem:-2} gives the desired contradiction.

\subsubsection{Conclusion of the proof of Lemma \ref{lem:iso}}

We have started with three cases, but only obtained two irreducible families.
This may sound a bit strange at first, but it will make sense momentarily.
For one thing, one easily checks that the $A_8+A_1$-configuration in 
Figure \ref{fig:321} is primitive;
indeed, one immediately detects a divisor $D$ of Kodaira type $I_8$, 
inducing an elliptic fibration, together with two smooth rational curves $C_1, C_2$
such that $C_i.D=1$, so the $I_8$ fiber is ramified and the curves are bisections,
leading to case \eqref{it2}.

On the hand, one can similarly verify that the Enriques surfaces from \ref{sss:6321}
support the root lattice $A_8+A_1$ both primitively and imprimitively.
This is related to the underlying structure of the singular fibers
induced by the $6$-torsion sections on $\Jac(Y)$.
Depending on how we set up the curves comprising $A_8$ relative to this structure,
they will embed primitively or not. \qed

\begin{Remark}
The different fibrations from \eqref{it1}, \eqref{it3} and \ref{sss:6321} can be connected 
by working out suitable divisors of Kodaira type on the respective graphs of (obvious) $(-2)$-curves,
but we shall not go into the details here.
\end{Remark}

\subsubsection{Moduli components}

For completeness and later reference, we point out that the remaining two moduli components are indeed distinct.
The proof will serve as a prototype for all other cases to follow.

\begin{Corollary}
\label{cor:A8+A1}
The root type $A_8+A_1$ is supported on two distinct families
of Enriques surfaces.
\end{Corollary}

\begin{proof}
The above discussion shows that there are at most two components.
It remains to prove that
 \eqref{it1} and \eqref{it2} are indeed not equivalent.
 For instance, this can be read off from the  transcendental lattices of the covering K3 surfaces,
 or already from their discriminants.
The former turns out to be generically
 $U+\langle 36\rangle$ for \eqref{it1},
 and  $U+\langle 4\rangle$ for \eqref{it2}
 (calculated using discriminant forms as laid out in Section \ref{s:latt}).
\end{proof}

\subsection{Root type $R=E_8+A_1$}

We have already worked out a family supporting the root type $R=E_8+A_1$ in \ref{ss:321}.
This could be obtained directly from $R$ by considering the isotropic vector $E=(e_8^\vee,0)+h$
which has exactly $R_0=E_7+A_1$ and a bisection $B'$ induced from $e_8$.
Indeed, this behaves exactly as in \ref{ss:321}.

\begin{Remark}
The family supporting the root type $E_8+A_1$ is remarkable since it comprises Enriques surfaces with finite automorphism group
(Kondo's type I in \cite{Kondo-Enriques}).
We remind the reader  that an Enriques surface has finite automorphism group
if and only if the set of smooth rational curves is non-empty, yet finite.
(Presently, there are exactly 12 smooth rational curves, i.e.~one more than displayed in Figure \ref{fig:321}.)
\end{Remark}

\subsection{Root type $D_6+3A_1$}
\label{sss2}
\label{sss:D_6+3A_1}

Proposition \ref{prop:base} and Table \ref{T1} tell us that
there is an elliptic fibration on $Y$ with 
root lattices $D_6+2A_1$ embedding into the singular fibers.
By Criterion \ref{crit},
the rational elliptic surface $\Jac(Y)$ is extremal.
Note that 
none of the resulting $(-2)$-divisors is met by the smooth rational bisection $B'$
(induced by a generator $a_1$ of the third summand $A_1$).
As before, this implies
that the reducible fibers
have types $I_2^*, I_2, I_2$,
i.e. $\Jac(Y) = X_{222}$,
and that a suitable composition of reflections $\sigma$ takes all the $(-2)$-curves generating $R_0$ to fiber components
(as stated in Proposition \ref{prop:curves}).
On the quadratic twist $S'$,
this leads to an integral section $P'$ of height $h(P')=4$ as in \ref{sss:0f},
but with the advantage of having a 2-torsion section, so that \ref{sss:2t} kicks in.
Indeed one can easily parametrise the quadratic twists with section $P'$ as required,
but as indicated in Proposition \ref{prop:unique}, this leads to two parametrizations,
one for each case from \ref{sss:2t}:
\[
x(P') = (t+\lambda)^2/4\lambda
\;\;\;
\text{ and } \;\;\; x(P') = (\mu^2t^2-2\mu(\mu-1)t+(\mu+1)^2)/4\mu.
\]
To see that these parametrizations give the same family of Enriques surfaces
(without having to determine them explicitly, in fact),
we switch to an auxiliary fibration as follows.
Consider the divisor $D=d_0+B'$ of Kodaira type $I_2$ on $Y$
where $d_0$ denotes the component of the original $I_2^*$ fiber met doubly by the bisection $B'$
(compare Figure \ref{fig:D6+3A1}).
Since $\sigma(d_2).D=1$,
$D$ is a half-pencil and $\sigma(d_2)$ a bisection for the elliptic fibration
induced by $|2D|$.
Naturally this has $D$ as a multiple singular fiber and the other singular fibers containing 
what's orthogonal to $D$:
\begin{eqnarray}
\label{eq:D4+3A1}
D_4=\langle \sigma(d_3),\hdots, \sigma(d_6)\rangle,\;\; A_1=\Z \sigma(d_1),\;\; 2A_1
\end{eqnarray}
where the last two summands stem from the original root type
(or equivalently from the two $I_2$ fibers -- the components not met by $B'$).
By criterion \ref{crit},
$\Jac(Y)$ is an extremal rational elliptic surface,
so we can use \cite{MP} to work out that it can only be $X_{222}$ again.
In particular, there is a fiber of type $I_2^*$ comprising $D_4$, two $A_1$'s and a connecting $(-2)$-curve.
Recall that the bisection $\sigma(d_2)$ meets the fiber in simple components;
by inspection, these have to be $\sigma(d_1), \sigma(d_3)$, since $\sigma(d_2)$ 
is perpendicular to all other $(-2)$-curves in \eqref{eq:D4+3A1}.
The next figure sketches the resulting configuration of smooth rational curves (up to exchanging $\sigma(d_5), \sigma(d_6)$).

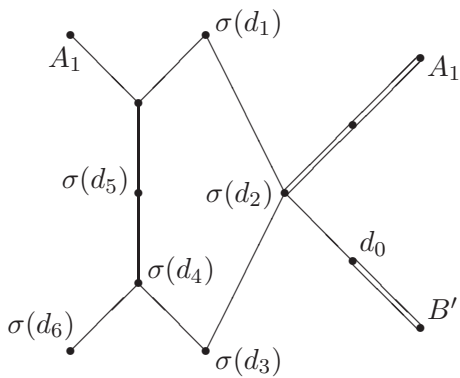
\begin{figure}[ht!]
\setlength{\unitlength}{.6mm}
\begin{picture}(90,78)(0,0)

\multiput(0,0)(30,0){2}{\circle*{1.5}}
\put(0,0){\line(1,1){15}}
\put(30,0){\line(-1,1){15}}

\put(-14,4){$\sigma(d_6)$}
\put(32,-4){$\sigma(d_3)$}

\multiput(0,70)(30,0){2}{\circle*{1.5}}

\put(0,70){\line(1,-1){15}}
\put(30,70){\line(-1,-1){15}}

\put(-5,63){$A_1$}
\put(32,71){$\sigma(d_1)$}

\multiput(15,15)(0,20){3}{\circle*{1.5}}
\put(15,15){\line(0,1){40}}

\put(17,16){$\sigma(d_4)$}

\put(-2,36){$\sigma(d_5)$}

\put(47.5,35){\circle*{1.5}}
\put(30,33){$\sigma(d_2)$}

\put(30,0){\line(1,2){17.5}}

\put(30,70){\line(1,-2){17.5}}

\put(47.5,35){\line(1,-1){15}}

\multiput(62.5,20)(15,-15){2}{\circle*{1.5}}
\put(64,22){$d_0$}
\put(79,7){$B'$}

\put(62.1,19.6){\line(1,-1){15}}
\put(62.9,20.4){\line(1,-1){15}}

\multiput(62.5,50)(15,15){2}{\circle*{1.5}}
\put(79,61){$A_1$}

\put(47.1,35.4){\line(1,1){30}}
\put(47.9,34.6){\line(1,1){30}}

%
%
%
%
%
%
%
%
%

%
%

\end{picture}
\caption{Induced fibration for $D_6+3A_1$}
\label{fig:D6+3A1}
\end{figure}

On the K3  cover $X$, the bisection $\sigma(d_2)$ splits into disjoint sections $O,P$.
Since there are correction terms  twice $3/2$ from the two $I_2^*$ fibers and $1$ from the ramified $I_4$ fiber,
we infer that $P$ has height zero.
That is, $P$ is two-torsion, so up to exchanging the two $I_2$ fibers, $x(P')$ is as in 
Table \ref{T2}.

\begin{Remark}
\label{rem:prim}
We point out that the $(-2)$-curves in Figure \ref{fig:D6+3A1}
also support the root types $E_7+2A_1$ and $A_7+2A_1$.
This is a rather special situation:
the root type $D_6+3A_1$  
automatically comes with the overlattice $E_7+2A_1$ (and with $A_7+2A_1$)
supported on $(-2)$-curves.
This is reflected in Table \ref{T2}
by listing for $R$ the same equations as for $E_7+2A_1$.
\end{Remark}

\subsection{Ramified representations}

The isotropic vectors in Table \ref{T1} were chosen in very  natural way relative to the gluing 
obtained from the discriminant groups.
In particular, these choices came with the extra benefit of leading
to one-dimensional families of Enriques surfaces with reduced fiber configurations
(a very natural result recorded in Proposition \ref{prop:ex}).
In terms of explicit equations, however, they are often quite complicated and thus hard to present.
For several root types (to be listed below),
Table \ref{T2} therefore includes  different equations
which make for a nicer representation.
In particular, they always involve a ramified fiber
which together with the two-torsion section is derived in the same way as in \ref{ss:321} (see especially Claim \ref{claim}).
For shortness, we just list the root types, the alternative isotropic vectors $\hat E$, the resulting root lattice $R_0$,
and the fiber configurations.

$$
\begin{array}{cccc}
R & \hat E & R_0 & \text{fiber configuration}\\
\hline
D_5+A_3+A_1 & d_2^\vee+h & 2A_3+2A_1 & 2I_4,I_4, I_2, I_2\\
E_8 + A_1 & e_8^\vee + h & E_7+A_1
& III^*, 2I_2, I_1\\
E_7+A_2 & e_3^\vee+h & A_5+A_2+A_1
& 2I_6, I_3, I_2, I_1\\
E_7+2A_1 & e_2^\vee+h & D_6+2A_1
& I_2^*, 2I_2, I_2\\
E_6+A_2+A_1 & e_1^\vee+h & A_5+A_2+A_1
& I_6, I_3, 2I_2, I_1
\end{array}
$$

\subsection{Proof of Theorem \ref{thm:any}}

We claim that the above considerations suffice to prove Theorem \ref{thm:any}.
To see this, recall from Corollary \ref{cor:by} that in most cases, $R_0$ determines the configuration of singular fibers,
and the Enriques surfaces can be found in Table \ref{T2}.
Otherwise, as recorded in Corollary \ref{cor:unique},
the configuration of singular fibers differs from the sum of the $\tilde R_v$,
or there are multiple fibers.
In either case, by Observation \ref{obs} 
the resulting $(-2)$-curve support a proper finite index overlattice of $R$
which the family files under.
\qed

\subsection{Proof of main part of Theorem \ref{thm}}
\label{ss:pf}

We emphasize that this proves all of Theorem \ref{thm} except for \textit{(ii)}
which will be the subject of the next section.

\section{Components of the moduli spaces}
\label{s:number}

In order to complete the proof of Theorem \ref{thm},
it remains to study the number of components of the moduli spaces
for each of the root types from Theorem \ref{thm:31}
which are not already covered by Corollary \ref{cor:unique}.
For an example, recall from Corollary \ref{cor:A8+A1}
that the root type $A_8+A_1$ is supported on two distinct families of Enriques surfaces:
one with reduced fiber configuration determined by the data in Table \ref{T1}, 
the other  associated to the overlattice $E_8+A_1$ (Kond\=o's family of type I).
In accordance with Theorem \ref{thm} \textit{(ii)},
we claim that this situation is not far from optimal:

\begin{Proposition}
\label{prop:number}
There are exactly two maximal root types which are
supported on three distinct one-dimensional families of Enriques surfaces:
\[
A_7+2A_1,\;\;\; 3A_3.
\]
All other moduli spaces of Enriques surfaces have either one or two components.
More precisely, those types with two moduli components are given in Table \ref{T3}.
\end{Proposition}

The following table collects the root types $R$ with exactly two moduli components.
It lists the orthogonal complement $R_0$ and  configurations of singular fibers with multiplicities,
both for $R$ and some overlattice $\hat R$ supported on the second family (cf.~Observation \ref{obs}).

\begin{table}[ht!]
$$
\begin{array}{|ccc|cc|}
\hline
R & R_0& \text{fibers} & \hat R & \text{fibers}\\
\hline
A_8+A_1 & A_7+A_1 & I_8, I_2 & E_8+A_1 & 2I_8, I_2\\
A_7 + A_2 & A_5+A_2+A_1 & I_6,I_3,I_2 & E_7 + A_2 & 2I_6, I_3, I_2\\
A_5+A_3+A_1 & 2A_3+2A_1 & I_4,I_4,I_2,I_2 & E_6+A_3 & I_1^*, I_4\\
A_5+2A_2 & A_5+A_2+A_1 & I_6,I_3,I_2 & E_7+A_2 & 2I_6,I_3,I_2\\
D_8+A_1 & A_7+A_1 & I_8, I_2 & E_8+A_1 & 2I_8, I_2
\\
D_6+A_3 & D_6+2A_1& I_2^*, I_2,I_2 & D_9 & I_4^*
\\
D_4+A_3+2A_1 & 2A_3+2A_1& I_4,I_4,I_2,I_2 
& D_7+2A_1 & 2I_4,I_4,I_2,I_2
\\
\hline
\end{array}
$$
\caption{Root types with exactly two moduli components}
\label{T3}
\end{table}

Note that Proposition \ref{prop:number} directly gives Theorem \ref{thm} \textit{(ii)},
thus completing the proof of Theorem \ref{thm} in view of \ref{ss:pf}.

The proof of Proposition \ref{prop:number} follows quite closely the lines of the previous sections,
but we have to take the different possible fiber configurations and multiplicities into account.
To most extent, this parallels what we did for root type $R=A_8+A_1$ in \ref{ss2:A_8+A_1},
so we only sketch the details for the two root types supported on 3 moduli components.

\subsection{Root type $R=A_7+2A_1$}
\label{ss:A7+2A1}
\label{ss:721}

For the root type $R=A_7+2A_1$, 
we have already seen two components of the moduli space:
on the one hand clearly the one displayed in Table \ref{T2} under $R$;
on the other hand related to the root type $D_6+3A_1$ (and $E_7+2A_1$) in Figure \ref{fig:D6+3A1},
compare Remark \ref{rem:prim}.
Furthermore, it is supported on
Kond\=o's family of type I,
i.e.~the family of root type $E_8+A_1$
as one can easily locate the root type $A_7+2A_1$
in the graph of the 12 $(-2)$-curves on these Enriques surfaces from \cite[Fig.~1.4]{Kondo-Enriques}.
We claim that these three families are distinct,
and that they comprise all Enriques surfaces supporting the root type $R$.

To exhibit a complete argument,
recall the data from table \ref{T1} with isotropic vector
\[
E = a_4^\vee + h, \;\;\; \text{ and } \;\;\; R_0=2A_3+2A_1.
\]
In the first instance, this leads to singular fibers given as $\tilde R_v$,
so $\Jac(Y)=X_{4422}$ with equations as in Table \ref{T2}.
The other configurations of singular fibers admitting an embedding of $R_0$ are the following:
\begin{eqnarray}
\label{eq:721}
\tilde D_5+\tilde A_3, \;\; \tilde D_6+2\tilde A_1, \;\; \tilde E_7+\tilde A_1, \;\; \tilde D_8, \;\; \tilde E_8,
\end{eqnarray}
but among these all except for $\tilde D_6+2\tilde A_1$ and $\tilde E_7+\tilde A_1$ can be ruled out
by arguing with the smooth rational bisection $B'$ induced by $a_4$.
In the case of $\tilde E_7+\tilde A_1$, we are exactly in the situation of \ref{ss:321},
so in particular $B'$ splits into zero section and two-torsion section 
and $Y$ supports the root type $E_8+A_1$.
It thus remains to study the case where $\Jac(Y)=X_{222}$.

\begin{Claim}
\label{claim2}
If $R_0\hookrightarrow \tilde D_5+\tilde A_3$, then
$B'$ splits into zero section and two-torsion section on the K3 cover.
\end{Claim}

\begin{proof}
Assume to the contrary that the section $P\in\MW(X)$ mapping to $B'$
is not two-torsion.
At any rate, it meets both $I_2^*$ fibers in far simple components
(not the one close to the identity component).
It follows that either there is a two-torsion section $Q$ meeting both fibers in the same component,
so that
\[
\langle P,Q\rangle = 2 - 2\cdot\frac 32 - \hdots \leq -1,
\]
or each two-torsion section $Q$ meets once the same and once a neighbouring component,
hence
\[
\langle P,Q\rangle = 2 - \frac 32 - 1 - \hdots \leq -\frac 12.
\]
In either case, the height pairing does not return zero despite $Q$ being torsion,
so we obtain the required contradiction.
\end{proof}

From Claim \ref{claim2}, it follows directly as in \ref{ss:321}
that one of the $I_2$ fibers is ramified.
Up to symmetry, $Y$ is thus given by the data $P=(0,0), q=t(t-\lambda)$,
as listed in Table \ref{T2} under the root types $D_6+3A_1, E_7+2A_1$.

It remains to prove that the three families of Enriques surfaces thus derived
are indeed distinct.
As natural as this may seem, this kind of statement is non-trivial as we have experienced
for the root type $D_6+3A_1$ in \ref{sss:D_6+3A_1}.
Presently, it can be achieved  by computing the discriminants 
of the generic N\'eron-Severi lattices of the covering K3 surfaces.
Here we can use \cite[\S 11 (22)]{SSh} to directly read off
discriminants $16$ for the second family and $4$ for the third.
Meanwhile the first family is endowed with a section $P$ mapping to $B'$
which generically meets four $I_4$ fibers in a component adjacent to the identity component.
Hence its height generically equals
\[
h(P) = 4 -3\cdot \frac 34 = 1.
\]
After verifying that  $P$ generically generates $\MW(X)$ together with the torsion sections
(i.e.~neither $P$ nor its translates by torsion sections are two-divisible
which can be checked both explicitly or abstractly),
formula \cite[\S 11 (22)]{SSh} again applies to calculate discriminant $64$.
Hence all three families are indeed distinct, 
and Proposition \ref{prop:number} for $R=A_7+2A_1$ follows.
\qed

\subsection{Root type $3A_3$}
\label{ss:3A3}

The root type $R=3A_3$ behaves quite similar to the previous one.
Especially, this appeals to the fact
that the isotropic vector $E=(a_2^\vee,0,0)+h/2$
yields the same perpendicular root lattice $R_0=2A_3+2A_1$.
The main difference lies in the intersection behaviour with the bisection $B'$
which presently meets the $A_1$ summands (as opposed to the $A_3$ summands).
This directly leads to two configurations of singular fibers,
\[
2\tilde A_3+2\tilde A_1, \;\;\; \tilde D_5+\tilde A_3
\]
(as reflected in Corollary \ref{cor:by});
in each case, there is a family with all fibers generically reduced,
but there is also another family with reduced $I_4$ fiber
(and $B'$ splitting off a two-torsion section on the K3 cover,
compare Table \ref{Ty}). 
Each family is easily parametrized (uniquely up to symmetry), see for instance, the entries in Table \ref{T2}
under root types $3A_3$ and $D_6+A_3$.

In order to distinguish moduli components,
we proceed by computing the corresponding generic discriminants of the K3 covers
using  \cite[\S 11 (22)]{SSh}:

$$
\begin{array}{ccc}
\Jac(Y) & \text{fiber configuration} & \text{discriminant}\\
\hline
X_{4422} & I_4,I_4,I_2,I_2 & 128\\
& 2I_4, I_4, I_2, I_2 & 32\\
X_{141} & I_1^*, I_4, I_1 & 32\\
& I_1^*, 2I_4, I_1 & 8
\end{array}
$$

Hence there are at least 3 moduli components,
and it remains to prove that the second and third families agree.
To this extent, we argue with the third family and the $(-2)$-curves which it is naturally endowed with:

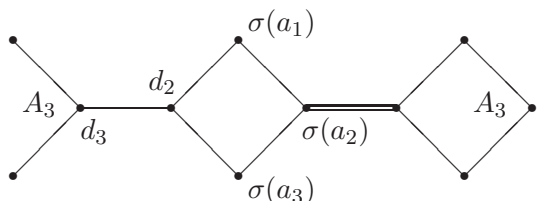
\begin{figure}[ht!]
\setlength{\unitlength}{.6mm}
\begin{picture}(120,35)(0,0)

\multiput(0,0)(50,0){2}{\circle*{1.5}}

\multiput(0,30)(50,0){2}{\circle*{1.5}}

\multiput(15,15)(20,0){2}{\circle*{1.5}}

\multiput(65,15)(20,0){2}{\circle*{1.5}}

\multiput(100,0)(0,30){2}{\circle*{1.5}}

\put(115,15){\circle*{1.5}}

\multiput(0,0)(50,0){3}{\line(1,1){15}}

\multiput(0,30)(50,0){3}{\line(1,-1){15}}

\multiput(50,0)(50,0){2}{\line(-1,1){15}}

\multiput(50,30)(50,0){2}{\line(-1,-1){15}}

\put(15,15){\line(1,0){20}}

\put(65,14.5){\line(1,0){20}}

\put(65,15.5){\line(1,0){20}}

\put(2,14){$A_3$}
\put(102,14){$A_3$}

\put(52,-5){$\sigma(a_3)$}

\put(52,32){$\sigma(a_1)$}

\put(64,8){$\sigma(a_2)$}

\put(15,8){$d_3$}

\put(30,19){$d_2$}
%
%
%
%
%
%
%
%
%
%
%
%
%
%
%
%
%
%
%
%
%
%

\end{picture}
\end{figure}

Note that this supports the root type $3A_3$,
but also $D_6+A_3$.
To extract the second family, consider the divisor $D=\sigma(a_1)+\sigma(a_2)+\sigma(a_3)+d_2$ of Kodaira type $I_4$.
Since it is met by the smooth rational curve $d_3$ with multiplicity one,
$D$ is a half-pencil and $d_3$ a  bisection.
Thus, by inspection of the figure, $|2D|$ exactly leads to the second configuration of singular fibers (with multiplicities).

%
%
%
%
%
%
%
%
%

We conclude that the root type $3A_3$ admits three moduli components as stated in
Proposition \ref{prop:number}. \qed

\subsection{Conclusion}
For the sake of brevity, we omit the analysis of the root types supported on two moduli components
(which follows the same lines of arguments).
This completes the proof of Theorem \ref{thm}. \qed

%
%
%
%
%

\subsection*{Acknowledgements}

Thanks to S\l awomir Rams and Chris Peters for  discussions on the subject
and to JongHae Keum for bringing this problem to my attention and helpful comments.


\begin{thebibliography}{99}

%
%

\bibitem{Barth} Barth, W.: \emph{K3 Surfaces with Nine Cusps.}
Geom. Dedic. {\bf 72} (1998), 171--178.



\bibitem{BHPV}
Barth, W., Hulek, K., Peters, C., van de Ven, A.:
\emph{Compact complex surfaces}.
Second edition,
Erg.~der Math.~und ihrer Grenzgebiete,
3.~Folge, Band {\bf 4}.~Springer (2004), Berlin.

\bibitem{BP}
Barth, W.,
Peters, C,:
\emph{Automorphisms of Enriques Surfaces},
Invent. Math. {\bf 73} (1983), 383--412.


\bibitem{Beauville} Beauville, A.:
\emph{Les familles stables de courbes elliptiques sur $\PP^1$
admettant quatre fibres singuli\'eres},
C.R.~Acad.~Sci.~Paris {\bf 294} (1982), 657--660.




%


\bibitem{Bourbaki}  
Bourbaki, N.: \emph{\'El\'ements de math\'ematique. Fasc. XXXIV. Groupes et alg\'ebres de Lie.}
Actualit\'es Scientifiques et Industrielles, No. {\bf 1337}. Hermann (1968), Paris.



%

\bibitem{CS}
Cartwright, D., Steger, T.:
\emph{Enumeration of the 50 fake projective planes}, 
Comptes Rendus Math. {\bf 348} (2010), 11--13.

%
%

%
%
%
%
%
%
%
%
%
%


\bibitem{HS}  Hulek, K., Sch\"utt, M.: \emph{Enriques surfaces and Jacobian elliptic surfaces},
Math. Z. {\bf 268} (2011), 1025--1056.


%
%
%

\bibitem{HK}
Hwang, D., Keum, JH:
\emph{The maximum number of singular points on rational homology projective planes}, J. Algebraic
Geom. 
{\bf 20} (2010), 495--523.


\bibitem{HKO}
Hwang, D, Keum, JH, Ohashi, H.:
\emph{Gorenstein $\Q$-homology projective planes},
Science China Math. {\bf 58} (2015), 501--512.



%

\bibitem{Keum-Viet}
Keum, JH:
\emph{The moduli space of $\Q$-homology projective planes with 5 quotient singular points}, Acta Math. Vietnamica {\bf 35}
(2010), 79--89.

\bibitem{K} Kodaira, K.:
\emph{On compact analytic surfaces I-III},
Ann.~of Math., {\bf 71} (1960), 111--152;
{\bf 77} (1963), 563--626; {\bf 78} (1963), 1--40.



\bibitem{Kondo-Enriques} Kond\=o, S.: \emph{Enriques surfaces with finite automorphism groups.}, Japan. J. Math. (N.S.) {\bf 12} (1986), 191--282. 


\bibitem{Kondo-Aut}
S.~Kond\=o, 
\emph{Automorphisms of algebraic K3 surfaces which act trivially
on Picard groups}, 
J.~Math.~Soc.~Japan, {\bf 44}, No.~1 (1992), 75--98.



\bibitem{MP} Miranda, R., Persson, U.:
\emph{On Extremal Rational Elliptic Surfaces},
Math.~ Z.~{\bf 193} (1986), 537--558.

\bibitem{Mumford}
Mumford, D.:
\emph{An algebraic surface with $K$ ample, $K^2 = 9, p_g = q = 0$}, 
Amer. J. Math. {\bf 101} (1979), 233--244.


\bibitem{Nikulin} Nikulin, V.~V.:
\emph{Integral symmetric bilinear forms and some of their applications},
Math.~USSR Izv.~{\bf 14}, No.~1 (1980), 103--167.
%
%
%
\bibitem{OS} Oguiso, K., Shioda, T.: \emph{The Mordell--Weil lattice of a rational elliptic surface}, Comment.~Math.~Univ.~St.~Pauli~{\bf 40} (1991), 83--99.
%
%


\bibitem{PSS} Piatetski-Shapiro, I.~I., Shafarevich, I.~R.: \emph{Torelli's theorem for algebraic surfaces of type ${\rm K}3$}, Izv.~Akad.~Nauk SSSR Ser.~Mat.~{\bf 35} (1971), 530--572.

\bibitem{PY}
Prasad, G., Yeung, S.-K.:
\emph{Fake projective planes}, 
Invent. Math. {\bf 168} (2007), 321--370;
addendum:  Invent. Math. {\bf 182} (2010), 213--227.


\bibitem{RS}
Rams, S., Sch\"utt, M.:
\emph{On Enriques surfaces with four cusps},
preprint (2014), arXiv: 1404.3924v2.

%
%
%
%
%
%
%

%
%
%
%
%
%
%
\bibitem{SSh} 
 Sch\"utt, M., Shioda, T.:
 \emph{Elliptic surfaces},
Algebraic geometry in East Asia - Seoul 2008, 
Advanced Studies in Pure Math.~{\bf 60} (2010), 51-160.  
%
%
%
%
%
%
%
\bibitem{ShMW} Shioda, T.: \emph{On the Mordell-Weil lattices}, Comm.~Math.~Univ.~St.~Pauli {\bf 39} (1990), 211--240.
%
%
%
%
%
%
%
%
\bibitem{Tate} Tate, J.: {\it Algorithm for determining the type
of a singular fibre in an elliptic pencil}, in: {\it Modular
functions of one variable IV} (Antwerpen 1972), Lect.~Notes in Math.~{\bf 476}
(1975), 33--52.

%


%

\bibitem{xiao}
Xiao, G.:
\emph{Galois covers between K3 surfaces},
Annales de l'institut Fourier {\bf 46} (1996), 73--88. 


\bibitem{Ye}
 Ye, Q.:
 \emph{On Gorenstein log del Pezzo surfaces}, 
 Japan J. Math. {\bf 28} (2002), 87--136.


\end{thebibliography}
\end{document}